\documentclass{article}

\usepackage{graphicx}
\usepackage[T1]{fontenc}
\usepackage[utf8]{inputenc}
\usepackage[dvipsnames]{xcolor}
\usepackage[breaklinks=true]{hyperref}

\newcommand{\eqISO}{\hyperref[eq:ISO]{(ISO(\ensuremath{\bidvector}))}\xspace}
\newcommand{\eqKKT}{\hyperref[eq:kktiso]{(KKT)}\xspace}
\newcommand{\eqD}{\hyperref[eq:ISO-dual-problem-zeta]{(D(\ensuremath{\pricevector}))}\xspace}
\newcommand{\eqPF}{\hyperref[eq:ISO-linear-production-problem]{(\ensuremath{P(\flowvector(\bidvector))})}\xspace}
\newcommand{\eqGNEP}{\hyperref[eq:equilibrium-problem-gnep]{(GNEP)}\xspace}

\usepackage{authblk}
\usepackage{url}
\usepackage{amsmath}
\usepackage{amssymb}
\usepackage{amsthm}
\usepackage{empheq}
\usepackage{relsize}
\usepackage{physics}
\usepackage{doi}
\usepackage{xspace}
\usepackage{pgfplots}
\usepackage{multirow}
\usepackage{booktabs}
\usepackage{makecell}
\usepackage{diagbox}
\pgfplotsset{compat=1.18}
\usepgfplotslibrary{groupplots}
\usepackage{pgfplotstable}
\usepackage{tikz}
\usetikzlibrary{positioning}
\usepackage{biblatex}
\makeatletter
\newcommand{\leqnomode}{\tagsleft@true\let\veqno\@@leqno}
\makeatother

\usepackage{todonotes}

\newtheoremstyle{bfnote}%
{}{}%
{\itshape}{}%
{\bfseries}{.}%
{ }%
{\thmname{#1}\thmnumber{ #2}\thmnote{ (#3)}}

\theoremstyle{bfnote}
\newtheorem{theorem}{Theorem}[section]
\newtheorem{corollary}{Corollary}[theorem]

\newtheorem{proposition}[theorem]{Proposition}

\theoremstyle{remark}
\newtheorem{remark}{Remark}
\newtheorem{example}[theorem]{Example}

\theoremstyle{plain}

\newcommand{\A}{\mathcal{A}}

\newcommand{\R}{\mathbb{R}}

\DeclareMathOperator{\sgn}{sgn}

\DeclareMathOperator*{\argmin}{argmin}

\renewcommand{\epsilon}{\varepsilon}

\newcommand{\nodeset}{\mathcal{N}}
\newcommand{\nodenumber}{N}
\newcommand{\node}{n}

\newcommand{\edgeset}{\mathcal{E}}
\newcommand{\edgenumber}{E}
\newcommand{\edge}{e}

\newcommand{\techset}{\mathcal{T}}
\newcommand{\technumber}{T}
\newcommand{\tech}{\tau}

\newcommand{\bidspace}{\A}

\newcommand{\productionletter}{q}
\newcommand{\nodalproductionletter}{Q}
\newcommand{\bidletter}{\alpha}
\newcommand{\flowletter}{f}
\newcommand{\resistanceletter}{r}
\newcommand{\priceletter}{\lambda}
\newcommand{\taxletter}{\zeta}
\newcommand{\kktcapacityletter}{\mu}
\newcommand{\kktflowletter}{\beta}
\newcommand{\isocostletter}{c}
\newcommand{\pollutionfactorletter}{z}
\newcommand{\pollutioncostletter}{\Lambda}
\newcommand{\demandletter}{d}
\newcommand{\pollutionlimitletter}{L}
\newcommand{\neighboursetletter}{K}

\newcommand{\constraintfunctionletter}{g}
\newcommand{\isoletter}{y}
\newcommand{\agentcostletter}{f}

\newcommand{\deviationletter}{x}

\newcommand{\pollutionlimit}{\overline{\pollutionlimitletter}}
\NewDocumentCommand{\demand}{O{\node}}{\demandletter_{#1}}
\NewDocumentCommand{\pollutionfactor}{O{\node} O{\tech}}{\pollutionfactorletter_{#1, #2}}
\NewDocumentCommand{\OTpollutionfactor}{O{\node}}{\pollutionfactorletter_{#1}}
\NewDocumentCommand{\pollutioncost}{O{\tech}}{\pollutioncostletter_{#1}}

\NewDocumentCommand{\resistance}{O{\edge}}{\resistanceletter_{#1}}
\NewDocumentCommand{\maximumflow}{O{\edge}}{\overline{\flowletter}_{#1}}
\NewDocumentCommand{\minimumflow}{O{\edge}}{\underline{\flowletter}_{#1}}
\NewDocumentCommand{\maximumcapacity}{O{\node} O{\tech}}{\overline{\productionletter}_{#1, #2}}
\NewDocumentCommand{\OTmaximumcapacity}{O{\node}}{\overline{\productionletter}_{#1}}
\NewDocumentCommand{\neighbourset}{O{\node}}{\neighboursetletter_{#1}}
\NewDocumentCommand{\isocost}{O{\node} O{\tech}}{\isocostletter_{#1, #2}}
\NewDocumentCommand{\OTisocost}{O{\node}}{\isocostletter_{#1}}
\NewDocumentCommand{\balanceconstraint}{O{\node}}{\constraintfunctionletter^{\text{B}}_{#1}}

\NewDocumentCommand{\capacityconstraint}{O{\node} O{\tech}}{\constraintfunctionletter^{\text{C}}_{#1, #2}}
\NewDocumentCommand{\capacityupperconstraint}{O{\node} O{\tech}}{\constraintfunctionletter^{\text{CU}}_{#1, #2}}
\NewDocumentCommand{\capacitylowerconstraint}{O{\node} O{\tech}}{\constraintfunctionletter^{\text{CL}}_{#1, #2}}

\newcommand{\activeupperset}{\text{NT}^{\text{CU}}}
\newcommand{\activelowerset}{\text{NT}^{\text{CL}}}

\newcommand{\pollutionfactorvector}{\boldsymbol{\pollutionfactorletter}}

\newcommand{\maximumcapacityvector}{\overline{\productionvector}}

\NewDocumentCommand{\production}{O{\node} O{\tech}}{\productionletter_{#1, #2}}
\NewDocumentCommand{\deviation}{O{\node} O{\tech}}{\deviationletter_{#1, #2}}
\NewDocumentCommand{\OTproduction}{O{\node}}{\productionletter_{#1}}
\NewDocumentCommand{\flow}{O{\edge}}{\flowletter_{#1}}
\NewDocumentCommand{\nodalproduction}{O{\node}}{\nodalproductionletter_{#1}}

\newcommand{\productionvector}{\boldsymbol{\productionletter}}
\newcommand{\flowvector}{\boldsymbol{\flowletter}}
\newcommand{\isovector}{\boldsymbol{\isoletter}}
\newcommand{\deviationvector}{\boldsymbol{\deviationletter}}

\NewDocumentCommand{\truecost}{O{\node} O{\tech}}{\hat{\bidletter}_{#1, #2}}
\NewDocumentCommand{\OTtruecost}{O{\node}}{\hat{\bidletter}_{#1}}
\NewDocumentCommand{\maximumbid}{O{\node} O{\tech}}{\overline{\bidletter}}
\NewDocumentCommand{\OTmaximumbid}{O{\node}}{\overline{\bideletter}_{#1}}

\newcommand{\truecostvector}{\hat{\boldsymbol{\bidletter}}}

\NewDocumentCommand{\agentcost}{O{\node}}{\agentcostletter_{#1}}

\NewDocumentCommand{\bid}{O{\node} O{\tech}}{\bidletter_{#1, #2}}
\NewDocumentCommand{\OTbid}{O{\node}}{\bidletter_{#1}}
\newcommand{\bidvector}{\boldsymbol{\bidletter}}
\NewDocumentCommand{\nodalbidvector}{O{\node}}{(\bid[#1][\tech])_{\tech}}

\newcommand{\tax}{\taxletter}
\NewDocumentCommand{\price}{O{\node}}{\priceletter_{#1}}
\NewDocumentCommand{\kktcapacityupper}{O{\node} O{\tech}}{\overline{\kktcapacityletter}_{#1, #2}}
\NewDocumentCommand{\OTkktcapacityupper}{O{\node}}{\overline{\kktcapacityletter}_{#1}}
\NewDocumentCommand{\kktcapacitylower}{O{\node} O{\tech}}{\kktcapacityletter_{#1, #2}}
\NewDocumentCommand{\OTkktcapacitylower}{O{\node}}{\kktcapacityletter_{#1}}
\NewDocumentCommand{\kktflowupper}{O{\edge}}{\overline{\kktflowletter}_{#1}}
\NewDocumentCommand{\kktflowlower}{O{\edge}}{\kktflowletter_{#1}}

\newcommand{\pricevector}{\boldsymbol{\priceletter}}
\newcommand{\kktcapacityuppervector}{\overline{\boldsymbol{\kktcapacityletter}}}
\newcommand{\kktcapacitylowervector}{\boldsymbol{\kktcapacityletter}}

\begin{document}

\title{Strategic Pricing in Electricity Markets with Pollution Constraints}

\author[1]{Luce Brotcorne}
\author[1]{Gaël Guillot}
\author[2]{Alejandro Jofré}
\author[2]{Benjamín Vera}
\affil[1]{Univ. Lille, Inria, CNRS, Centrale Lille}
\affil[2]{Depto de Ingeniería Matemática and CMM, Universidad de Chile}
\date{Submitted for publication on April 27th, 2026}
\setcounter{Maxaffil}{0}
\renewcommand\Affilfont{\itshape\small}

\maketitle

\begin{abstract}
	We introduce a new model for a regulated day-ahead type auction-based electrical market in which the system operator can measure and limit the producer's emissions when choosing its optimal dispatch. We prove properties of this market model that describe it as a generalization to other previous works in the electrical market literature. Furthermore, we use standard sensitivity analysis tools to measure the impact of these pollution variables on market equilibria and present numerical examples of this effect.
\end{abstract}


\section{Introduction} \label{sect:introduction}

Electricity generation is one of the largest contributors to global greenhouse gas emissions, accounting for nearly 30\% of global emissions as of 2021\footnote{World Resources Institute, Mengpin Ge, Johannes Friedrich, and Leandro Vigna, Where Do Emissions Come From? 4 Charts Explain Greenhouse Gas Emissions by Sector, December 5, 2024. Available at \href{https://www.wri.org/insights/4-charts-explain-greenhouse-gas-emissions-countries-and-sectors}{world resources institute}
. Accessed July 14, 2025.}. Thus, decarbonizing the energy sector is essential for meeting international climate targets. In recent years, global electricity production has increasingly shifted towards renewable sources \cite{irena2025capacity}, a trend that poses technical and economic challenges for grid operation.

Since the widespread liberalization of electricity markets in the 1980's \cite{kirschen2018fundamentals}, energy dispatch decisions are no longer centrally planned but rather emerge from market mechanisms that match supply and demand. In these markets, firms behave strategically, and regulatory interventions, such as pollution taxes or emissions caps, interact in complex ways with market dynamics. Furthermore, physical features of the transmission network (such as resistance-induced energy losses) introduce locational disparities and market imperfections that cannot be captured by simplified models. In some real-world systems, such as the Chilean electrical grid, transmission losses can account for up to 30\% of the total energy transmitted \cite{AtrapandoElSol}, and their presence directly enables expressions of market power, as demonstrated in \cite{escobar2010monopolistic}. Understanding how environmental policy instruments such as emissions limits or carbon pricing may influence the existing market structure is therefore a pressing concern for both regulators and researchers.

Electricity market modeling is a highly active research area. Several studies incorporate transmission losses in oligopolistic settings \cite{bjorndal2005deregulated,aussel2013electricity,DidierAusselDeregulated,escobarjofreStanford}. Others adopt machine learning techniques over a simplified network model, such as \cite{Francesco2024Learning}, who use an actor-critic deep reinforcement learning framework. Integer programming approaches to strategic bidding have been documented in \cite{kostarelou2021,kozanidis2023exact}. Environmental considerations have also entered this modeling space. For example, \cite{hernandez2022pollutionregulation} studies pollution regulation in a continuous-time setting. However, to the best of our knowledge, no existing model fully integrates multiple generation technologies allowing for strategic bidding, resistance-induced network losses enabling imperfect competition, and nodal pricing within a bilevel equilibrium framework. This paper aims to fill-in that gap by providing a tractable but flexible model that captures all these features.

We consider a nodal electricity market with a standard two-layer structure: producers submit bids to a central dispatching agent (the Independent System Operator (ISO)) who solves an optimization problem to assign production and manage power flows. This setup is modeled as a multi-leader-single-follower game, in which strategic producers (leaders) interact through a shared ISO (follower), leading to a Nash equilibrium with bilevel structure.

Our framework supports both pay-as-bid and pay-as-clear schemes, but we adopt the latter, in line with European practice (see \cite{ACER2022}). Under pay-as-clear, each producer is compensated at a nodal price that reflects the marginal cost of delivering energy at its location. Nodal pricing, first introduced in \cite{scweppe1988spotpricing}, has become a key component of electricity market design and is widely used in practice\footnote{See EPEX SPOT, \textbf{Basics of the Power Market}, available at
\url{https://www.epexspot.com/en/basicspowermarket} (accessed June 21, 2024).}. In simplified models with no transmission losses, nodal prices reduce to the highest accepted bid (see \cite{Francesco2024Learning}). However, in realistic networks with resistive losses and environmental penalties, these prices diverge and encode richer information about the physical and economic cost of delivery.

This work provides theoretical insights as well as numerical tools for analyzing pollution-constrained market equilibria. Our main contributions are outlined as follows. First, we establish that, under mild assumptions, the ISO's dispatch problem admits a unique solution (both primal and dual) for any given bid profile by the generating agents. This property facilitates strategic bidding problems as it is no longer necessary to distinguish optimistic and pessimistic formulations. Next, we harness this uniqueness property to reformulate the multi-leader-follower game as a mixed complementarity problem (MCP) using first-order optimality conditions. Finally, we implement a grid-search procedure to obtain (and analyze the variation of) equilibrium outcomes under different pollution control policies. Our model supports sensitivity analysis over policy parameters, enabling us to identify price shocks, shifts in production patterns, and emergent expressions of market power. In particular, we compare the effects of emissions caps and pollution penalties on dispatch and cost outcomes.

The rest of this paper is organized as follows. In Section \ref{sect:market-model}, we introduce the market model, notation and assumptions that are used throughout this work. Section \ref{sect:uniqueness} presents the uniqueness theorems that constitute the theoretical centerpiece of this work, this theorems are followed by results giving a qualitative description of the ISO's electrical dispatch solution. In Section \ref{sect:solution-approach}, the computational approach for finding market equilibria is outlined. This is followed by numerical results on selected network instances in Section \ref{sect:examples}, which also presents discussion on algorithmic performance. Concluding remarks and possible future directions ar discussed in Section \ref{sect:conclusion}.

\section{Electrical market model}\label{sect:market-model}

We study an electricity market structured over a directed graph \((\nodeset, \edgeset)\), where nodes represent generating agents and edges represent a total of \(\edgenumber\) transmission lines. Each agent \(\node \in \nodeset\) controls a set of \(\technumber\) generation technologies \(\tech \in \techset\) and submits bids \(\bidvector_\node = (\bid)_{\tech \in \techset}\) in a sealed-bid auction process. The feasible bidding set for agent \(\node\) is thus defined defined as a multi-dimensional box \(\bidspace_\node = \prod_{\tech \in \techset} [\truecost, \maximumbid]\) where \(\truecost\) is the agent's (private) true cost of energy production via technology \(\tech\), and \(\maximumbid\) is a public upper bound defined by market rules.

The Independent System Operator (ISO) acts as a centralized dispatcher, which responds to the collective bid \(\bidvector = (\bidvector_\node)_{\node \in \nodeset}\) by making generation and transmission decisions in order to minimize operation cost while satisfying network balance, pollution and capacity constraints. That is, by solving the following convex optimization problem: 
\begin{subequations} \label{eq:ISO}
	\begin{align}
		\text{ISO}(\bidvector): \;
    & \min_{\productionvector,\flowvector}
    && \sum_{\node \in \nodeset, \tech \in \techset} (\bid + \pollutioncost \pollutionfactor)\production
    \label{ISO:objective} \\
    &\;\, \text{s.t.} \quad
    && \sum_{\edge \in \neighbourset} \frac{\resistance}{2}\flow^2 + \demand
    \leq
    \sum_{\tech \in \techset} \production \!\!
    + \!\! \sum_{\edge \in \neighbourset} \flow \sgn(\node, \edge),
    && \forall \node,
    \label{ISO:balance} \\
    &&&  \sum_{\node \in \nodeset, \tech \in \techset} \pollutionfactor \production
    \leq \pollutionlimit,
    && 
    \label{ISO:pollution} \\
    &&& 0 \leq \production \leq \maximumcapacity,
    && \forall \node, \tech.
    \label{ISO:capacity}
	\end{align}
\end{subequations}
Here, \(\production\) denotes production from node \(\node\) and technology \(\tech\), and \(\flow\) denotes power flow on edge \(\edge\). Each technology emits pollution at rate \(\pollutionfactor\), and the ISO penalizes pollution coming from technology \(\tech\) through a cost coefficient \(\pollutioncost\), leading to objective function \eqref{ISO:objective} which may be written as \(\sum_{\node, \tech} \isocost \production\) if we define the net cost by \(\isocost = \bid + \pollutioncost \pollutionfactor\).

Constraint \eqref{ISO:balance} is the balance constraint: each edge \(\edge = (\node_{\text{out}}, \node_{\text{in}})\) has resistance \(\resistance\), inducing a quadratic loss of \(\resistance \flow^2\), equally split between connected nodes. Defining the sign function \(\sgn(\node, \edge)\) by \(\sgn(\node_{\text{in}}, \edge) = 1\), \(\sgn(\node_{\text{out}}, \edge) = -1\) and letting \(\neighbourset\) be the set of edges \(\edge \in \edgeset\) which are neighbouring the node \(\node\), we get equation \eqref{ISO:balance} by imposing that the loss and local demand do not exceed the local production and net inbound flow. These constraints have associated dual variables denoted by \(\price \geq 0\), interpreted as \textit{nodal prices} and measured in units of \([\$/E]\). Economically, \(\price\) represents the marginal value of injecting one additional unit of energy at the corresponding node.

Constraint \eqref{ISO:pollution} is the pollution constraint. It means that total emissions must remain below the limit \(\pollutionlimit\). The associated multiplier \(\tax \geq 0\), measured in \([\$/\text{CO}^2]\), can be interpreted as the marginal utility to the ISO of relaxing the pollution limit by one unit, or equivalently, the marginal cost of reducing emissions by one unit. Finally, constraint \eqref{ISO:capacity} is the capacity constraint, imposing an upper bound \(\maximumcapacity\) for every node \(\node\) and technology \(\tech\). They are associated with dual variables \(\kktcapacitylower, \kktcapacityupper \geq 0\), also measured in $[\$/E]$.

Given a collective bid \((\bidvector_1, \dots, \bidvector_\nodenumber)\) and a corresponding ISO solution \((\productionvector, \flowvector)\) with associated multipliers \((\pricevector, \tax, \kktcapacitylowervector, \kktcapacityuppervector)\), we define the utility perceived by agent \(\node\) as
\begin{equation} \label{eq:agent-utility-function}
	\agentcost (\productionvector, \flowvector, \pricevector, \tax, \kktcapacitylowervector, \kktcapacityuppervector) = \sum_{\tech \in \techset} \qty[(\price - \pollutioncost \pollutionfactor) - \pollutionfactor \tax - \truecost]\production.
\end{equation}
Here, \(\pollutionfactor \tax\) accounts for the tax incurred when the pollution constraint \eqref{ISO:pollution} is binding. Also, the term \(\price - \pollutioncost \pollutionfactor\) represents the agent's effective revenue per unit of energy: The nodal price \(\price\) received from the ISO, adjusted to remove the component reflecting environmental cost in \eqISO. Note that the dependence of the ISO's primal response \((\productionvector, \flowvector)\), as well as its dual response \((\pricevector, \tax, \kktcapacitylowervector, \kktcapacityuppervector)\), on the bids \((\bidvector_1, \dots, \bidvector_\nodenumber)\), is not explicitly given and may not be uniquely defined. However, our uniqueness results in Theorems \ref{thm:ISO-primal-uniqueness} and \ref{thm:ISO-dual-uniqueness}, along with a suitable tie-breaking rule (see Remark \ref{rmk:tie-breaking}), will enable us to make these responses unique, thus making the utility expressions given in equation \eqref{eq:agent-utility-function} implicit functions of the bid profile \(\bidvector\). This defines a normal form strategic bidding game \((\bidspace_{\node}, \agentcost)_{\node\in \nodeset}\) whose Nash equilibria are the main object of study in this work.

Some assumptions on the problem data are in order. First, we will assume problem \eqISO satisfies Slater's condition: There exists a pair \((\productionvector, \flowvector)\) such that all inequality constraints \eqref{ISO:balance}, \eqref{ISO:pollution}, and \eqref{ISO:capacity} are satisfied with strict inequality. This mild assumption is typically met in applied contexts and ensures that the KKT conditions for the ISO problem provide both necessary and sufficient conditions for optimality, which will be essential in the developments to follow.

We also assume that all the marginal production costs \(\truecost\) as well as the resistances \(\resistance\) are strictly positive, and that there exists at least one node \(\node\) with nontrivial demand \(\demand > 0\). Another reasonable assumption is that the pollution factors belonging to every node \((\pollutionfactor)_\tech\) are pairwise distinct. This allows us to differentiate among technologies and is realistic, as identical emissions profiles are uncommon in practice.

Finally, we introduce a competitiveness condition: For every \(\node\), we impose
\[\demand - \sum_{\edge \in \neighbourset} \frac{1}{2 \resistance} < 0.\]
This expresses a form of local substitutability: the rest of the network can fully supply node \(\node\) if needed. It plays a role in limiting local market power and prevents situations where an agent could profitably increase bids due to inelastic ISO demand.

Table \ref{tab:notation} provides a summary of the notation introduced in this section. We also now list, for future reference, the KKT conditions for problem \eqISO.
\begin{subequations}
	\begin{empheq}[left=\text{KKT} \!:\! \empheqlbrace]{alignat=3}
		\bid + \pollutioncost \pollutionfactor - \price + \tax \pollutionfactor
		+ \kktcapacityupper - \kktcapacitylower &= 0,
							&& \forall \node, \tech \label{KKT-nodes} \\[0.4em]
							\sum_{\node : \edge \in \neighbourset}
		\price(\resistance \flow - \sgn(n, e)) &= 0,
						       && \forall \edge \label{KKT-edges} \\[0.4em]
						       \demand + \!\!\!
						       \sum_{\edge \in \neighbourset} \!\!\!
						       \qty(\tfrac{\resistance}{2}\flow^2 - \flow \sgn(n, e))
						       - \sum_\tech \production
		\leq 0 &\perp \price \geq 0,
		       && \forall \node \label{CS-balance} \\[0.4em]
		       \sum_{\node, \tech} \pollutionfactor \production - \pollutionlimit
		\leq 0 \;&\perp\; \tax \geq 0,
			 && \label{CS-pollution} \\[0.4em]
		\production \geq 0 \;&\perp\; \kktcapacitylower \geq 0,
				     && \forall \node, \tech \label{CS-capacity-lower} \\[0.4em]
		\production \leq \maximumcapacity \;&\perp\; \kktcapacityupper \geq 0,\;
						    && \forall \node, \tech \label{CS-capacity-upper}
	\end{empheq}
	\label{eq:kktiso}
\end{subequations}

\begin{table}[ht]
	\centering
	\begin{tabular}{lll}
		\toprule
		\textbf{Category} & \textbf{Symbol} & \textbf{Description} \\
		\midrule
		\multirow{3}{*}{Sets and totals}
				  & $\nodeset, \nodenumber$ & Set of nodes, indexed by $\node \in \nodeset$. \(\nodenumber := |\nodeset|\) \\
				  & $\edgeset, \edgenumber$ & Set of edges, indexed by $\edge \in \edgeset$. \(\edgenumber := |\edgeset|\) \\
				  & $\techset, \technumber$  & Set of technologies, indexed by $\tech \in \techset$. \(\technumber := |\techset|\) \\
				  \midrule
				  \multirow{6}{*}{Network data}
				  & $\truecost$ & Unit production cost of technology $\tech$ at node $\node$ \\
				  & $\maximumcapacity$ & Production capacity of technology $\tech$ at node $\node$ \\
				  & $\demand$ & Energy demand at node $\node$ \\
				  & $\resistance$ & Resistance of transmission line $\edge$ \\
				  & $\maximumbid$ & Maximum allowable bid \\
				  & $\sgn(\node,\edge)$ & Orientation of edge $\edge$ w.r.t.\ node $\node$ \\
				  \midrule
				  \multirow{3}{*}{Pollution data}
				  & $\pollutionfactor$ & Unit pollution emitted by technology $\tech$ at node $\node$ \\
				  & $\pollutioncost$ & Social cost per unit of pollution from technology $\tech$ \\
				  & $\pollutionlimit$ & Global pollution limit \\
				  \midrule
				  Agent variables
				  & $\bid$ & Bid cost of technology $\tech$ submitted by agent at node $\node$ \\
				  \midrule
				  \multirow{2}{*}{ISO variables}
				  & $\production$ & Dispatched production of technology $\tech$ at node $\node$ \\
				  & $\flow$ & Dispatched power flow on line $\edge$ \\
				  \midrule
				  \multirow{3}{*}{ISO multipliers}
				  & \(\price\) & Associated to \eqref{ISO:balance}, nodal price of energy at \(\node\) \\
				  & \(\tax\) & Associated to \eqref{ISO:pollution}, global tax on emissions \\
				  & \(\kktcapacitylower, \kktcapacityupper\) & Associated to \eqref{ISO:capacity}, lower and upper bounds \\
				  \bottomrule
	\end{tabular}
	\caption{Summary of notation for the electricity market model.}
	\label{tab:notation}
\end{table}

\section{Uniqueness of ISO response} \label{sect:uniqueness}

We now present a series of propositions leading to Theorem \ref{thm:ISO-primal-uniqueness}, which establishes conditions on the bid profile \(\bidvector\) under which the resulting solution \((\productionvector, \flowvector)\) is unique. The theorem is followed by a uniqueness of multipliers result in Theorem \ref{thm:ISO-dual-uniqueness} and together, these primal-dual uniqueness results are key in order to establish the relevance of our electrical market formulation.

From Lemma 2.2 of \cite{DidierAusselDeregulated}, for any \((\productionvector, \flowvector)\) solution to \eqISO, the following natural bound holds on the flows:
\begin{equation} \label{eq:natural-flow-bound}
	|\flow| \leq \frac{1}{\resistance}.
\end{equation}
Proposition \ref{prop:price-positivity} means that if inequalities \eqref{eq:natural-flow-bound} are strict, then the nodal prices are positive across the network.
\begin{proposition}[Price positivity] \label{prop:price-positivity}
	Let \((\productionvector, \flowvector, \pricevector, \tax, \kktcapacitylowervector, \kktcapacityuppervector)\) be a solution to \eqKKT such that for all \(\edge \in \edgeset\), \(|\resistance \flow| < 1\), then \(\forall \node \in \nodeset: \price > 0\).
\end{proposition}
\begin{proof}
	Strictly positive demand results in at least one node needing to produce energy. Thus, there exists \(\node_0, \tech_0 \in \nodeset \times \techset\) with \(\production[\node_0][\tech_0] > 0\) and therefore \(\kktcapacitylower[\node_0][\tech_0] = 0\). Evaluating the equation \eqref{KKT-nodes} for this pair we have that
	\[\price[\node_0] = \bid[\node_0][\tech_0] + \pollutioncost[\tech_0] \pollutionfactor[\node_0][\tech_0] + \tax \pollutionfactor[\node_0][\tech_0] + \kktcapacityupper[\node_0][\tech_0] \geq \bid[\node_0][ \tech_0] \geq \truecost[\node_0][\tech_0] > 0\]
	so \(\price[\node_0] > 0\). Additionally, let \(\edge = (\node_{\text{out}}, \node_{\text{in}}) \in \edgeset\) be an arbitrary edge. Evaluating the equation \eqref{KKT-edges} for \(\edge\) and rearranging, we get
	\[\price[\text{out}] = \frac{1 - \resistance \flow}{\resistance \flow + 1} \price[\text{in}].\]
	By the hypothesis that \(|\resistance \flow| < 1\), we have that if one of the prices is strictly positive, then the other one is too. By connectedness of the network, this concludes the result.
	\qed
\end{proof}

From this result as well as the complementary slackness conditions \eqref{CS-balance}, Kirchhoff's law can be easily obtained.

\begin{corollary}[Kirchhoff's law] \label{Kirchhoff-law}
	For any $(\productionvector, \flowvector)$, solution to \eqISO with \(|\resistance \flow| < 1\) for all \(\edge \in \edgeset\), one has that for any node $\node$:
	\begin{equation} \label{Kirchhoff-law-equation}
		\sum_{\tech \in \techset} \production = \demand + \sum_{\edge \in \neighbourset} \qty( \frac{\resistance}{2}\flow^2 - \flow \sgn(\node, \edge) ).
	\end{equation}
	That is, at optimality, the balance constraints are always active.
\end{corollary}

\begin{remark}
	This result can be obtained under more general conditions such as a network with flow bounds \(\flow \in [- \maximumflow, \maximumflow]\) or without the requirement that \(|\resistance \flow| < 1\). This can be done by modifying the argument given in lemma 3.1 in \cite{aussel2013electricity}.
\end{remark}

Having Kirchhoff's law established, uniqueness of flows can be easily obtained by leveraging strict convexity.

\begin{proposition}[Flow uniqueness] \label{prop:flow-uniqueness}
	Let \((\productionvector_1, \flowvector_1)\), \((\productionvector_2, \flowvector_2)\) be solutions to \eqISO. If \(\forall \edge \in \edgeset : \resistance > 0\), then \(\flowvector_1 = \flowvector_2\).
\end{proposition}

\begin{proof}
	Suppose that \(\flowvector_1 \neq \flowvector_2\). Since the problem is convex, the midpoint \((\tilde{\productionvector}, \tilde{\flowvector}) = \qty(\frac{\productionvector_1 + \productionvector_2}{2}, \frac{\flowvector_1+\flowvector_2}{2})\) solves \eqISO. But since \(\resistance > 0\), the functions \(x \mapsto \frac{\resistance}{2} x^2\) are strictly convex so \((\tilde{\productionvector}, \tilde{\flowvector})\) does not satisfy equation \eqref{Kirchhoff-law-equation}. That contradicts corollary \ref{Kirchhoff-law}.
	\qed
\end{proof}

Proposition \ref{prop:flow-uniqueness} allows us to define the mapping \(\bidvector \mapsto \flowvector(\bidvector)\) which, to every bid profile \(\bidvector\), assigns the corresponding ISO flow solution. By Kirchhoff's law, the uniqueness of productions is already guaranteed at this step if \(\technumber = 1\). So in the following we assume \(\technumber \geq 2\). If we define
\[\nodalproduction(\flowvector) := \demand + \sum_{\edge \in \neighbourset}\qty(\frac{\resistance}{2}\flow^2 - \flow \sgn(\node, \edge)),\]
then we can deduce a linear production problem that is solved by \(\productionvector\) given the unique flow solution: If \((\productionvector, \flowvector(\bidvector))\) is a solution to \eqISO, then \(\productionvector\) solves
\begin{subequations} \label{eq:ISO-linear-production-problem}
	\begin{align}
		P(\flowvector(\bidvector)): \;
    & \min_{\productionvector}
    && \sum_{\node, \tech}
    (\bid + \pollutioncost \pollutionfactor)
    \production
    \\
    &\;\, \text{s.t.} \quad
    && \sum_{\tech} \production
    \geq \nodalproduction(\flowvector(\bidvector)),
    && \forall \node \in \nodeset,
    \label{eq:ISO-linear-production-problem-balance} \\
    &&& \sum_{\node, \tech} \pollutionfactor \production
    \leq \pollutionlimit,
    &&
    \label{eq:ISO-linear-production-problem-pollution} \\
    &&& 0 \leq \production \leq \maximumcapacity,
    && \forall \node \in \nodeset, \tech \in \techset.
    \label{eq:ISO-linear-production-problem-capacity}
	\end{align}
\end{subequations}

Next, we provide conditions under which the solution to problem \eqPF is unique. Our arguments make use of one of the characterizations provided by the Mangasarian LP uniqueness theorem (see Theorem 2, item (iii) of \cite{Mangasarian1979Uniqueness}), which states that a solution \(\overline{x}\) to the linear programming problem
\begin{align*}
	\min_{x \in \mathbb{R}^n} &c^\top x \\
				  & Ax \leq b, \\
				  & C x = d,
\end{align*}
is unique whenever there is no \(x \neq 0\) such that
\[
	Cx = 0, \qquad A_{J(\overline{x})} x \leq 0, \qquad c^\top x \leq 0.
\]
Here, \(J(\overline{x})\) denotes the set of indices for the inequality constraints which are active at \(\overline{x}\). Thus, we need to consider separately the cases for a solution with active and inactive pollution constraint. In the remainder, for a production solution \(\productionvector\), we define the sets
\begin{align*}
	\activeupperset (\productionvector) &= \{ (\node, \tech) \in \nodeset \times \techset : \production = \maximumcapacity \}, \\
	\activelowerset (\productionvector) &= \{ (\node, \tech) \in \nodeset \times \techset : \production = 0 \}.
\end{align*}
That is, the node-technology indices where \(\productionvector\) activates its upper or lower bounds.

\begin{proposition}[Production uniqueness, inactive pollution case] \label{prop:production-uniqueness-inactive}
	Let \(\productionvector\) be a solution to \eqPF for bids \(\bidvector\). Suppose that \(\sum_{\node, \tech} \pollutionfactor \production < \pollutionlimit\) and that the nodal net costs \((\isocost)_\tech\) are all distinct for every node \(\node\). Then, \(\productionvector\) is unique.
\end{proposition}

\begin{proof}
	Let \(\deviationvector = (\deviation)_{\node, \tech}\) be such that
	\begin{align}
		\sum_{\node, \tech} \isocost \deviation & \leq 0, \tag{I1} \label{eq:mang-inactive-objective} \\
		\forall n : \sum_{\tech} \deviation & = 0, \tag{I2} \label{eq:mang-inactive-balance} \\
		\forall \node, \tech \in \activeupperset (\productionvector) : \deviation & \leq 0, \tag{I3a} \label{eq:mang-inactive-capacity-upper} \\
		\forall \node, \tech \in \activelowerset (\productionvector) : \deviation & \geq 0, \tag{I3b} \label{eq:mang-inactive-capacity-lower}
	\end{align}
	and define, for every \(\node\),
	\[
		\mathcal{P}_{\node} = \qty{ \tech \in \techset : \deviation > 0}, \qquad \mathcal{N}_{\node} = \qty{\tech \in \techset : \deviation < 0}.
	\]
	Note that by \eqref{eq:mang-inactive-balance}, \(\sum_{\tech \in \mathcal{P}_{\node}} \deviation = - \sum_{\tech \in \mathcal{N}_{\node}} \deviation\). Thus, either both \(\mathcal{P}_{\node}, \mathcal{N}_{\node}\) are empty or both are nonempty. For \(\node\) in the latter case, define
	\[
		m_{\node} = \min_{\tech \in \mathcal{P}_{\node}} \isocost, \qquad M_{\node} = \max_{\tech \in \mathcal{N}_{\node}} \isocost.
	\]
	Suppose now by contradiction that \(x \neq 0\). We will prove that
	\[
		\exists \node : M_{\node} \geq m_{\node}.
	\]
	Indeed, on the contrary we would have that \(M_{\node} < m_{\node}\) for all \(\node\), resulting in
	\begin{align*}
		\sum_{\node, \tech} \isocost \deviation &= \sum_{\node} \qty( \sum_{\tech \in \mathcal{P}_{\node}} \isocost \deviation + \sum_{\tech \in \mathcal{N}_{\node}} \isocost \deviation ) \\
							&\geq \sum_{\node} \qty( m_{\node} \sum_{\tech \in \mathcal{P}_{\node}} \deviation + M_{\node} \sum_{\tech \in \mathcal{N}_{\node}} \deviation ) \\
							& = \sum_{\node} \underbrace{(m_{\node} - M_{\node})}_{>0} \underbrace{\sum_{\tech \in \mathcal{P}_{\node}} \deviation}_{>0} > 0.
	\end{align*}
	Since this contradicts \eqref{eq:mang-inactive-balance}, the claim follows. Now, by hypothesis, since the \(\isocost\) are distinct, the claim holds with strict inequality. That is, there exists a node \(\node\) as well as \(\tech^+ \in \mathcal{P}_{\node}, \tech^- \in \mathcal{N}_{\node}\) such that \(\isocost[\node][\tech^+] < \isocost[\node][\tech^-] \). By using \eqref{eq:mang-inactive-capacity-upper}, \eqref{eq:mang-inactive-capacity-lower}, we then get that
	\[
		(\node, \tech^+) \notin \activeupperset (\productionvector), \qquad (\node, \tech^-) \notin \activelowerset (\productionvector).
	\]
	This means that a small enough modification of \(\productionvector\) augmenting \(\production[\node][\tech^+]\) and equally diminishing \(\production[\node][\tech^-]\) has strictly lower cost, contradicting the optimality of \(\productionvector\). Thus, \(\deviationvector=0\) and the uniqueness follows from Mangasarian's theorem.
	\qed
\end{proof}

We now consider the case of a production solution \(\productionvector\) with active pollution constraint. The argument given in the last proof is not sufficient for this case as the availability of a technology with lower cost does not contradict optimality whenever said technology is more polluting and the limit is binding. We therefore have to avoid a situation where the ISO can redistribute pollution while maintaining cost, both among technologies of the same agent (internal ties) and between different agents (external ties). This is encoded in the conditions of the next proposition.

\begin{proposition}[Production uniqueness, active pollution case] \label{prop:production-uniqueness-active}
	Let \(\productionvector\) be a solution to \eqPF for bids \(\bidvector\). Suppose that \(\sum_{\node, \tech} \pollutionfactor \production = \pollutionlimit\) and that the following conditions hold:
	\begin{enumerate}
		\item The nodal net costs \((\isocost)_\tech\) are all distinct for every node \(\node\).
		\item For all \(\node, \tech_1, \tech_2\), if \(a^*, b^*\) is the unique solution of the linear system
			\begin{align*}
				a - \pollutionfactor[\node][\tech_1] b &= \isocost[\node][\tech_1], \\
				a - \pollutionfactor[\node][\tech_2] b &= \isocost[\node][\tech_2],
			\end{align*}
			then for every \(\tech \neq \tech_1, \tech_2\), \(\isocost \neq a^* - \pollutionfactor b^*\).
		\item For every technology group \((\node_1, \tech_1), (\node_1, \tech_2), (\node_2, \tech_3), (\node_2, \tech_4)\) with \(\tech_1 \neq \tech_2\) and \(\tech_3 \neq \tech_4\), we have that
			\[
				\frac{\isocost[\node_1][\tech_1] - \isocost[\node_1][\tech_2] }{\pollutionfactor[\node_1][\tech_2] - \pollutionfactor[\node_1][\tech_1] } \neq \frac{\isocost[\node_2][\tech_3] - \isocost[\node_2][\tech_4] }{\pollutionfactor[\node_2][\tech_4] - \pollutionfactor[\node_2][\tech_3] }.
			\]
	\end{enumerate}
	Then, \(\productionvector\) is unique.
\end{proposition}

\begin{proof}
	Let \(\deviationvector = (\deviation)_{\node, \tech}\) be such that
	\begin{align}
		\sum_{\node, \tech} \isocost \deviation & \leq 0, \tag{A1} \label{eq:mang-active-objective} \\
		\forall \node : \sum_{\tech} \deviation &= 0, \tag{A2} \label{eq:mang-active-balance} \\
		\forall \node, \tech \in \activeupperset (\productionvector) : \deviation & \leq 0, \tag{A3a} \label{eq:mang-active-capacity-upper} \\
		\forall \node, \tech \in \activelowerset (\productionvector) : \deviation & \geq 0, \tag{A3b} \label{eq:mang-active-capacity-lower} \\
		\sum_{\node, \tech} \pollutionfactor \deviation & \leq 0. \tag{A4} \label{eq:mang-active-pollution}
	\end{align}
	Since \(\productionvector\) is a solution to \eqPF, there exist multipliers \(\pricevector, \tax, \kktcapacityuppervector, \kktcapacitylowervector \geq 0\) such that the following first order condition is satisfied
	\[
		\forall \node, \tech : \isocost + \tax \pollutionfactor - \price + \kktcapacityupper - \kktcapacitylower = 0.
	\]
	This means depending on the technology \(\node, \tech\), we have the following cases:
	\[
		\begin{cases}
			\isocost + \tax \pollutionfactor - \price \leq 0, & \node, \tech \in \activeupperset (\productionvector), \\
			\isocost + \tax \pollutionfactor - \price \geq 0, & \node, \tech \in \activelowerset (\productionvector), \\
			\isocost + \tax \pollutionfactor - \price = 0, & \text{ else.}
		\end{cases}
	\]
	From this as well as conditions \eqref{eq:mang-active-capacity-upper} and \eqref{eq:mang-active-capacity-lower}, we deduce
	\begin{equation} \label{eq:claim-uniqueness-positivity}
		\forall \node, \tech : (\isocost + \tax \pollutionfactor - \price) \deviation \geq 0
	\end{equation}
	which, in particular, implies
	\[
		\sum_{\node, \tech} (\isocost + \tax \pollutionfactor) \deviation \geq \sum_{\node, \tech} \price \deviation.
	\]
	Now, note that by \eqref{eq:mang-active-objective}, \eqref{eq:mang-active-pollution} and \eqref{eq:mang-active-balance},
	\begin{gather*}
		0 \geq \sum_{\node, \tech}\isocost \deviation = \sum_{\node, \tech} (\isocost + \tax \pollutionfactor) \deviation - \tax \sum_{\node, \tech} \pollutionfactor \deviation \\
		\geq \sum_{\node, \tech} (\isocost + \tax \pollutionfactor) \deviation \geq \sum_{\node, \tech} \price \deviation = \sum_\node \price \sum_{\tech} \deviation = 0.
	\end{gather*}
	Therefore, all inequalities are binding, meaning that \(\sum_{\node, \tech} \isocost \deviation = 0\) and
	\[
		\sum_{\node, \tech} (\isocost + \tax \pollutionfactor - \price) \deviation = 0.
	\]
	Along with condition \eqref{eq:claim-uniqueness-positivity}, we get that all the terms in the sum must be zero. We conclude the following complementarity condition:
	\begin{equation} \label{eq:claim-uniqueness-complementarity}
		\forall \node, \tech : (\isocost + \tax \pollutionfactor - \price = 0) \vee (\deviation = 0).
	\end{equation}
	Having this, we proceed by contradiction. If \(\deviationvector \neq 0\), there is a node \(\node_1\) and, by \eqref{eq:mang-active-balance}, at least two technologies \(\tech_1, \tech_2\) with \(\deviation[\node_1][\tech_1] , \deviation[\node_1][\tech_2] \neq 0\). Using \eqref{eq:claim-uniqueness-complementarity}, we obtain the following linear system
	\begin{align*}
		\price[\node_1] - \tax \pollutionfactor[\node_1][\tech_1] &= \isocost[\node_1][\tech_1] , \\
		\price[\node_1] - \tax \pollutionfactor[\node_1][\tech_2] &= \isocost[\node_1][\tech_2] .
	\end{align*}
	Solving for \(\tax\), we get the unique solution
	\[
		\tax = \frac{\isocost[\node_1][\tech_1] - \isocost[\node_1][\tech_2] }{\pollutionfactor[\node_1][\tech_2] - \pollutionfactor[\node_1][\tech_1] }.
	\]
	By hypothesis 2, these equalities cannot hold for any other \(\tech \neq \tech_1, \tech_2\), thus by \eqref{eq:claim-uniqueness-positivity}, the rest of the \(\deviation[\node_1][\tech] \) must be zero. Having this, we evaluate \eqref{eq:mang-active-objective} which we know holds with equality
	\[
		\sum_{\node \neq \node_1} \sum_{\tech} \isocost \deviation + \deviation[\node_1][\tech_1] (\isocost[\node_1][\tech_1] - \isocost[\node_1][\tech_2] ) = 0.
	\]
	By hypothesis 1, the second term is nonzero. Therefore another node \(\node_2 \neq \node_1\) has two technologies \(\tech_3 \neq \tech_4\) with \(\deviation[\node_2][\tech_3] , \deviation[\node_2][\tech_4] \neq 0\). Evaluating \eqref{eq:claim-uniqueness-complementarity} for these technologies, we get a similar linear system for \((\price[\node_2], \tax)\) having the solution
	\[
		\tax = \frac{\isocost[\node_2][\tech_3] - \isocost[\node_2][\tech_4] }{\pollutionfactor[\node_2][\tech_4] - \pollutionfactor[\node_2][\tech_3] },
	\]
	which contradicts the earlier solution in light of hypothesis 3. We conclude \(\deviationvector = 0\) and the uniqueness follows by Mangasarian's theorem.
	\qed
\end{proof}

From Propositions \ref{prop:flow-uniqueness}, \ref{prop:production-uniqueness-inactive} and \ref{prop:production-uniqueness-active}, the following theorem immediately follows:

\begin{theorem}[ISO primal uniqueness] \label{thm:ISO-primal-uniqueness}
	Let \((\productionvector, \flowvector)\) be a solution to \eqref{eq:ISO} for bids \(\bidvector\). If the following conditions hold:
	\begin{enumerate}
		\item For all \(\edge \in \edgeset\), \(\resistance > 0\).
		\item The nodal net costs \((\isocost)_\tech\) are all distinct for every node \(\node\).
		\item For all \(\node, \tech_1, \tech_2\), if \(a^*, b^*\) is the unique solution of the linear system
			\begin{align*}
				a - \pollutionfactor[\node][\tech_1] b &= \isocost[\node][\tech_1], \\
				a - \pollutionfactor[\node][\tech_2] b &= \isocost[\node][\tech_2],
			\end{align*}
			then for every \(\tech \neq \tech_1, \tech_2\), \(\isocost \neq a^* - \pollutionfactor b^*\).
		\item For every technology group \((\node_1, \tech_1), (\node_1, \tech_2), (\node_2, \tech_3), (\node_2, \tech_4)\) with \(\tech_1 \neq \tech_2\) and \(\tech_3 \neq \tech_4\), we have that
			\[
				\frac{\isocost[\node_1][\tech_1] - \isocost[\node_1][\tech_2] }{\pollutionfactor[\node_1][\tech_2] - \pollutionfactor[\node_1][\tech_1] } \neq \frac{\isocost[\node_2][\tech_3] - \isocost[\node_2][\tech_4] }{\pollutionfactor[\node_2][\tech_4] - \pollutionfactor[\node_2][\tech_3] }.
			\]
	\end{enumerate}
	then \((\productionvector, \flowvector)\) is unique.
\end{theorem}

\begin{remark}[Tie-breaking rule] \label{rmk:tie-breaking}
	Whenever the conditions for Theorem \ref{thm:ISO-primal-uniqueness} are not met, we still know that the flow solution is unique and that the optimal production set can then be written as a bounded polyhedron in \(\R^{\nodenumber \technumber}\) given by the solution set of the linear problem \eqPF. Therefore, for completeness, let us introduce the convention that whenever indifferent, the ISO will choose as its solution the unique corresponding flow along with the midpoint of the optimal polyhedron calculated as the equally weighted average of its vertices.
\end{remark}

\begin{example}
	For an example where ISO uniqueness does not hold due to external ties, consider the two-node electrical network depicted below, where the pollution cost is set to \(\pollutioncost = 1\) for both technologies and the pollution limit is set to \(\pollutionlimit = 5\). Since the network is symmetric, the solution must have zero flow, resulting in every node satisfying its own demand. However, even though the preferred technology (in terms of net cost \(\isocost\)) is the second one, it cannot be used in both nodes, as that would violate the pollution limit. The ISO must then distribute the emitted pollution among the nodes, but is indifferent as to how to do it. Clearly, the last hypothesis of Theorem \ref{thm:ISO-primal-uniqueness} does not hold in this example.
	
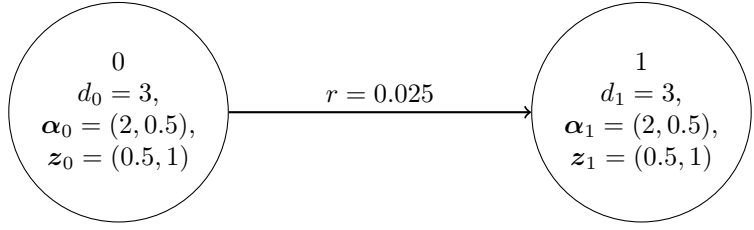
\begin{figure}[htbp]
		\centering
		\begin{tikzpicture}[node distance=4cm]
			\node[draw, circle, align=center] (node0) at (0,0) {
					0 \\
					$\demand[0] = 3$, \\
					$\bidvector_0 = (2, 0.5)$, \\
					$\pollutionfactorvector_0 = (0.5, 1)$
				};

			\node[draw, circle, align=center, right=of node0] (node1) {
					1 \\
					$\demand[1] = 3$, \\
					$\bidvector_1 = (2, 0.5)$, \\
					$\pollutionfactorvector_1 = (0.5, 1)$
				};

			\draw[->, thick] (node0.east) -- (node1.west) 
				node[midway, above, align=center] {$\resistanceletter = 0.025$};
		\end{tikzpicture}
	\end{figure}
\end{example}

We now discuss the uniqueness of multipliers \((\pricevector, \tax, \kktcapacitylowervector, \kktcapacityuppervector)\) associated to a given solution \((\productionvector, \flowvector)\) to the ISO problem. By Slater's condition, we get the Mangasarian Fromovitz constraint qualification, which is in turn equivalent to the multipliers being bounded (see \cite{peterson1973qualifications}). But uniqueness is only implied by the stronger Linear Independence Constraint Qualification condition (LICQ), which requires a priori assumptions on the ISO solution \((\productionvector, \flowvector)\) similar to the ones presented in Proposition 2.6 from \cite{DidierAusselDeregulated}. Those assumptions do not generally hold in many relevant network instances. Thus, one might wish for a more general result ensuring global multiplier uniqueness (that is, uniqueness of multipliers for every possible bid combination). We obtain one by making use of convex duality principles in what follows.

\begin{theorem}[Uniqueness of Multipliers] \label{thm:ISO-dual-uniqueness}
	If there is no selection \(S \subseteq N \times T\) such that \(\pollutionlimit = \sum_{(\node, \tech) \in S} \pollutionfactor \maximumcapacity\), then the Lagrange multipliers associated to the ISO's problem are unique for every \(\bidvector \in \bidspace\).
\end{theorem}
\begin{proof}
	We know from classical optimization theory that KKT multipliers for \eqISO are solutions to its associated Lagrange dual problem, which is given by
	\begin{align*}
    & \max_{\pricevector, \tax, \kktcapacitylowervector, \kktcapacityuppervector \geq 0} 
    && \!\!\!\!\!\!\!\sum_{e = (n, m) \in \edgeset} \frac{2}{\resistance} \frac{\price[n] \price[m]}{\price[n] + \price[m]}\! +\! \sum_{\node} \!\!\qty(\demand\! -\!\!\! \sum_{\edge \in \neighbourset} \frac{1}{2 \resistance}) \price - \sum_{\node, \tech} \maximumcapacity \kktcapacityupper - \pollutionlimit \tax \\
    & \quad\; \text{s.t.} \quad
    && \bid + \pollutioncost \pollutionfactor - \price + \tax \pollutionfactor + \kktcapacityupper - \kktcapacitylower = 0, \quad \forall \node \in \nodeset, \tech \in \techset .
	\end{align*}
	We may eliminate the variables \(\kktcapacitylower\) as they are slack variables. With that, the maximization over \(\kktcapacityupper\) is also trivial and results in
	\[
		\kktcapacityupper = \max \qty{0, \price - \isocost - \tax \pollutionfactor}.
	\]
	Maximizing then over \(\tax\) results in the following one-dimensional minimization problem:
	\begin{equation}\label{eq:ISO-dual-problem-zeta}
		\text{D}(\pricevector) : \min_{\tax \geq 0} g(\tax) := \pollutionlimit \tax + \sum_{\node, \tech} \maximumcapacity \max\qty{0, \price - \isocost - \tax \pollutionfactor}. \notag
	\end{equation}
	This is a convex piecewise linear minimization problem. So we may use the known subgradient optimality condition \(0 \in \partial g(\tax^*)\). The subgradient can be computed using the Moreau-Rockafellar theorem (see Theorem 23.8 in \cite{rockafellar1997convex}), the computation gives
	\[
		\partial g(\tax) = \pollutionlimit - \sum_{\node, \tech : \tax < \tax_{\node, \tech}^*} \maximumcapacity \pollutionfactor - \sum_{\node, \tech : \tax = \tax_{\node, \tech}^*} [0, \pollutionfactor] \maximumcapacity.
	\]
	where the breakpoints are given by
	\[
		\tax_{\node, \tech}^* = \frac{\price - \isocost}{\pollutionfactor}.
	\]
	This subdifferential takes the form of a non-decreasing step function outside the breakpoints. The step values are given by \(\pollutionlimit - \sum_{\node, \tech \in S} \maximumcapacity \pollutionfactor\) for some set \(S \subseteq \nodeset \times \techset\). By the hypothesis, \(0\) is never one of these constant values, so the multiplier \(\tax^*\) is necessarily \(0\) or one of the breakpoints \(\tax_{\node, \tech}^*\), and by monotonicity of the step function, the uniqueness of \(\tax\) in terms of \(\pricevector\) follows. By a standard result in convex analysis (see e.g. Proposition 3.3.1 in \cite{bertsekas2009}), we get that the (negative) value of problem \eqD is concave as a function of \(\pricevector\). The uniqueness of \(\pricevector\) itself follows then immediately by the strict concavity of the first term in the dual problem. This concludes the proof.
	\qed
\end{proof}

We now derive qualitative properties that must hold for the solutions of \eqISO satisfying the uniqueness assumptions of Theorems \ref{thm:ISO-primal-uniqueness} and \ref{thm:ISO-dual-uniqueness}. These properties provide a description of how the ISO solutions can behave at optimality and, by extension, at equilibrium. First, observe that, if the pollution limit is not active at a solution \(\productionvector\), then problem \eqPF decomposes into \(\nodenumber\) problems of nodal production dispatch. These problems have the known structure of \textit{fractional knapsack problems} (see e.g Chapter 17 of \cite{KorteVygenCombinatorial}) and therefore, the form of their solutions is well known. From this, the following result can be easily obtained.

\begin{proposition}[Activation of production bounds: Inactive pollution case] \label{prop:technology-activation-inactive-pollution}
	Let $(\productionvector, \flowvector)$ be a solution to \eqISO for \(\bidvector\) satisfying the hypothesis for Proposition \ref{prop:production-uniqueness-inactive}. Assume $(\productionvector, \flowvector)$ does not activate the pollution constraint. Then, for every node $\node$ there exists at most one technology $\tech_\node$ with $\production[\node][\tech_\node] \in (0, \maximumcapacity[\node][\tech_\node])$.
\end{proposition}

In the other case (that is, when the pollution constraint is active) we do not have a classical problem to compare against, but we can make use of ideas from the proof of Proposition \ref{prop:production-uniqueness-active} to get the following result:

\begin{proposition}[Activation of production bounds: Active pollution case] \label{prop:technology-activation-active-pollution}
	Let \((\productionvector, \flowvector)\) be a solution to \eqISO for \(\bidvector\) satisfying the hypotheses for Proposition \ref{prop:production-uniqueness-active}. Assume \((\productionvector, \flowvector)\) activated the pollution constraint. Then,
	\begin{enumerate}
		\item All nodes have at most two technologies not activating any of their bounds.
		\item If there is a node that has two technologies not activating any of their bounds, then that node is unique.
	\end{enumerate}
\end{proposition}

\begin{proof}
	If any of the two properties did not hold, we may evaluate the KKT conditions \eqref{KKT-nodes} to get incompatible values of the multiplier \(\tax\). This yields a contradiction since these values are assumed to be distinct. \qed
\end{proof}

\section{Solution approach} \label{sect:solution-approach}

In what follows, we define \(\isovector = (\productionvector, \flowvector, \pricevector, \tax, \kktcapacitylowervector, \kktcapacityuppervector)\) as the full ISO response to the agents' bids \(\bidvector\) and group the collection of equalities in \eqKKT (that is, \eqref{KKT-nodes}, \eqref{KKT-edges}, \eqref{CS-balance}, \eqref{CS-pollution}, \eqref{CS-capacity-lower}, \eqref{CS-capacity-upper}) as a single vector equation in the form \(H(\bidvector, \isovector) = 0\) as well as the dual feasibility constraints (that is, \eqref{ISO:balance}, \eqref{ISO:pollution}, \eqref{ISO:capacity} and \(\pricevector, \tax, \kktcapacitylowervector, \kktcapacityuppervector \geq 0\)) into a vector inequality which we rearrange as \(G(\isovector) \leq 0\). With this notation and by Slater's condition, we know that \(\isovector\) is a response to \(\bidvector\) if and only if \(H(\bidvector, \isovector) = 0, G(\isovector) \leq 0\). Thus, the market equilibrium problem can be written as:
\begin{subequations}
	\begin{empheq}[left=\text{GNEP} : \empheqlbrace]{align*}
    &\text{Find } \qty(\bidvector_1^*, \dots, \bidvector_N^*, \isovector^*) \text{ satisfying } \\
    &\begin{aligned}
	    \quad \forall \node : (\bidvector_\node, \isovector^*) \in 
	    \argmin_{\bidvector_\node, \isovector} \quad &\agentcost(\isovector) \\
	    \text{s.t.} \quad
							 &\bidvector_\node \in \bidspace_\node, \\
							 &H(\bidvector_\node, \bidvector_{-\node}^*, \isovector) = 0, \notag \\
							 &G(\isovector) \leq 0.
    \end{aligned}
	\end{empheq} \label{eq:equilibrium-problem-gnep}
\end{subequations}

This is a Generalized Nash Equilibrium Problem (GNEP) with \(\isovector\) acting as a shared variable among all agents (see \cite{facchinei2010generalized}). According to the previous Theorems \ref{thm:ISO-primal-uniqueness} and \ref{thm:ISO-dual-uniqueness} as well as Remark \ref{rmk:tie-breaking}, the ISO response mapping \(\bidvector \mapsto \isovector(\bidvector)\) is well defined. Under stronger assumptions, the Fiacco \& McCormick theorem (see Theorem 3.1 on \cite{fiacco2006sensitivity}) may be used in order to conclude local differentiability of this response mapping. This, along with using the Implicit Function Theorem to get an explicit form of \(\grad{y}(\bidvector)\), opens the possibility of first-order numerical approaches based on descent techniques that could yield convergence to Nash equilibria, this remains a promising topic for a future work. As a first approach for now, our numerical experiments are done by attempting to solve a KKT reformulation of problem \eqGNEP directly.

To compute equilibria, we use the \texttt{GAMS} modelling environment\footnote{Code for reproducing the experiments in this paper is available at the following repository: \href{https://github.com/Benja-Vera/strategic-pricing-pollution-constraints}{https://github.com/Benja-Vera/strategic-pricing-pollution-constraints}}, which currently provides the only implementation of the Extended Mathematical Programming (EMP) framework. However, equilibrium problems with bilevel structure such as the one considered in this work are not natively supported by EMP. To address this, we represent the system \eqKKT as not a mixed complementarity problem but via its nonlinear product reformulation. This results in a GNEP in which one copy of the ISO solution and multipliers \(\isovector_\node\) is assigned to every agent as a decision variable constrained by \eqKKT. However, in order to mitigate the number of variables, the \textit{switching} strategy discussed in \cite{FERRIS2009EMP} is employed. This modifies the resulting MCP so that the shared ISO variables are not replicated. Possible discrepancies between agents are eliminated by theorems \ref{thm:ISO-primal-uniqueness} and \ref{thm:ISO-dual-uniqueness}.

Because of non-convexity, not all MCP solutions are true Nash equilibria. Therefore, we search from a variety of initial guesses. Resulting MCP solutions are then filtered by an \textit{agent optimality test}: for each agent \(\node\), we solve an NLP corresponding to agent \(n\)'s best response according to problem \eqGNEP. If no agent has a profitable unilateral deviation beyond a numerical tolerance \(\epsilon_{\text{AO}}\), the candidate is accepted as a Nash Equilibrium. The search was carried out using a grid-based strategy. Each bid variable was discretized into \(K\) points, yielding \(K^{\nodenumber \technumber}\) combinations. Each bid configuration was paired with its corresponding ISO response \(\isovector(\bidvector)\) and used as an initial guess for solving the MCP.

The underlying solver for the MCP reformulation built via the EMP framework was \texttt{KNITRO} \cite{Byrd2006Knitro}. Although preliminary tests were conducted using \texttt{PATH}, a solver widely applied in complementarity problems (see \cite{Ferris1995PATH}), its performance in finding solutions to the MCP in this context was found to be insufficient. Similar limitations have also been reported in related literature, such as \cite{bautista2007formulation}. \texttt{KNITRO} was also employed to solve the agents' best response problems during the agent optimality test. For the ISO's dispatch problem \eqISO, associated with a given bid vector \(\bidvector\), the \texttt{GUROBI} solver was used, as it is well suited for handling the convex quadratic structure of the ISO's problem.

The sensitivity analysis involved repeated calls to the grid search routine, each with slight changes to parameters such as the pollution limit \(\pollutionlimit\) or pollution costs \(\pollutioncost\). To reduce the computational burden, two implementation strategies were used. First, a \textit{prior searching mode} reused the last verified equilibrium as a starting point in the next MCP solve, helping track equilibrium shifts across small parameter changes. Second, a \textit{unique solution mode} allowed early termination of the search once a single valid equilibrium was found, significantly reducing computation time, particularly in large instances.

\section{Numerical results} \label{sect:examples}

The numerical results to be presented rely on four network instances which are introduced in increasing order of complexity in subsection \ref{sub:instances}. In \ref{sub:sensitivity-analysis}, we then present the results of the numerical sensitivity experiments. We conclude the section by making a summary of the computational performance of the grid searching algorithm on the presented experiments.

\subsection{Problem instances} \label{sub:instances}

The first network, depicted in Figure \ref{fig:instance-1-asymmetric}, is adapted from a setting originally studied in \cite{escobar2010monopolistic}, where a closed form expression for the market equilibrium of a symmetric network was derived. After validating our numerical procedure in the symmetric case, we modify the instance by framing it as competition between neighbouring nodes, each with a different generating portfolio. In this instance we only study the effect of the pollution cost mechanism and correspondingly, we set the limit \(\pollutionlimit\) to a sufficiently high value.

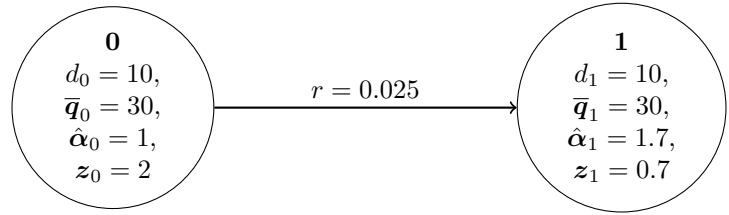
\begin{figure}[htbp]
	\centering
	\begin{tikzpicture}[node distance=4cm]

		\node[draw, circle, align=center] (node0) at (0,0) {
				$\boldsymbol{0}$ \\
				$\demand[0] = 10$, \\
				$\maximumcapacityvector_0 = 30$, \\
				$\truecostvector_0 = 1$, \\
				$\pollutionfactorvector_0 = 2$
			};

		\node[draw, circle, align=center, right=of node0] (node1) {
				$\boldsymbol{1}$ \\
				$\demand[1] = 10$, \\
				$\maximumcapacityvector_1 = 30$, \\
				$\truecostvector_1 = 1.7$, \\
				$\pollutionfactorvector_1 = 0.7$
			};

		\draw[->, thick] (node0.east) -- (node1.west) 
			node[midway, above, align=center] {$\resistanceletter = 0.025$};

	\end{tikzpicture}
	\caption{Two-node one-technology network. Here, node \(0\) acts as a generator possessing a cheap but highly polluting energy source, whereas node \(1\)'s technology is significantly less polluting but has a higher production cost. For both nodes, it is feasible to satisfy the total network demand.}
	\label{fig:instance-1-asymmetric}
\end{figure}

As it provides only a single generation technology, Instance \ref{fig:instance-1-asymmetric} is not well suited to our setting, which requires agents to make strategic choices between multiple technologies. To address this, we introduce a new technology for every agent, thus extending into the network depicted in Figure \ref{fig:instance-2}.
\begin{figure}[htbp]
	\centering
	\begin{tikzpicture}[node distance=4cm]
		\node[draw, circle, align=center] (node0) at (0,0) {
				$\boldsymbol{0}$ \\
				$\demand[0] = 10$, \\
				$\maximumcapacityvector_0 = (30, 30)$, \\
				$\truecostvector_0 = (1.5, 1)$, \\
				$\pollutionfactorvector_0 = (1, 2)$
			};

		\node[draw, circle, align=center, right=of node0] (node1) {
				$\boldsymbol{1}$ \\
				$\demand[1] = 10$, \\
				$\maximumcapacityvector_1 = (30, 30)$, \\
				$\truecostvector_1 = (1.5, 1)$, \\
				$\pollutionfactorvector_1 = (1, 2)$
			};

		\draw[->, thick] (node0.east) -- (node1.west) 
			node[midway, above, align=center] {$\resistanceletter = 0.025$};
	\end{tikzpicture}
	\caption{Two-node two-technology network. The trade-off presented in Instance \ref{fig:instance-1-asymmetric} is now internal to each node. As before, the capacities enable each node to fulfill the complete network demand on its own if required.}
	\label{fig:instance-2}
\end{figure}
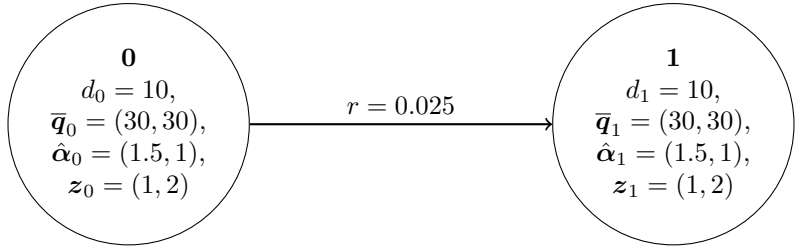

Here, every agent has two equal options: A clean but expensive technology in \(\tech = 0\), and a cheaper but more polluting option in \(\tech = 1\). This symmetry is expected to induce ill-conditioning of the resulting market solutions as it results in symmetric bids at equilibrium, which may induce non-uniqueness of the ISO response. Thus, we compare the results of Instance \ref{fig:instance-2} with a slight perturbation which is depicted in Figure \ref{fig:instance-2-asymmetric}.

\begin{figure}[htbp]
	\centering
	\begin{tikzpicture}[node distance=4cm]
		\node[draw, circle, align=center] (node0) at (0,0) {
				$\boldsymbol{0}$ \\
				$\demand[0] = 10$, \\
				$\maximumcapacityvector_0 = (30, 30)$, \\
				$\truecostvector_0 = (1.55, 0.95)$, \\
				$\pollutionfactorvector_0 = (0.95, 2.1)$
			};

		\node[draw, circle, align=center, right=of node0] (node1) {
				$\boldsymbol{1}$ \\
				$\demand[1] = 10$, \\
				$\maximumcapacityvector_1 = (30, 30)$, \\
				$\truecostvector_1 = (1.5, 1)$, \\
				$\pollutionfactorvector_1 = (1, 2)$
			};

		\draw[->, thick] (node0.east) -- (node1.west) 
			node[midway, above, align=center] {$\resistanceletter = 0.025$};
	\end{tikzpicture}
	\caption{Asymmetric version of Instance \ref{fig:instance-2}. Only values corresponding to pollution and production cost were changed.}
	\label{fig:instance-2-asymmetric}
\end{figure}
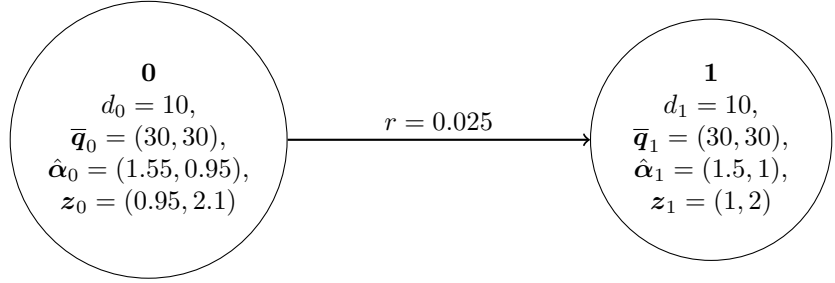

The final example aims to test our approach on a multi-nodal network with asymmetric production capacities where agent interactions are more complex than in the previously studied two-node settings. The network configuration depicted in Figure \ref{fig:network-and-table} is inspired from the Chilean electricity market, following the setup in \cite{hernandez2022pollutionregulation}. Key characteristics of this instance are summarized in the given table. Unlike the earlier examples, this network does not satisfy a simple \(n-1\) robustness assumption: if production at the central node was unavailable, the northern and southern regions could not meet their own demand and that of the center. This structural vulnerability is expected to have an influence on equilibrium outcomes.

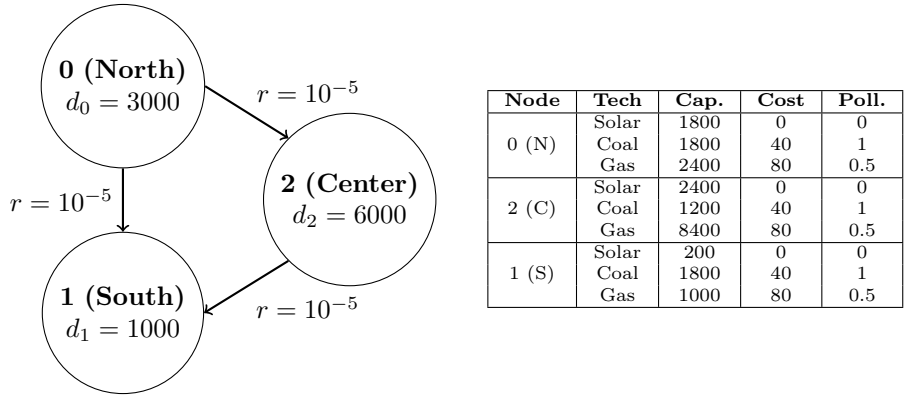
\begin{figure}[htbp]
	\centering
	\begin{minipage}[c]{0.48\textwidth}
		\centering
		\begin{tikzpicture}[node distance=4cm]
			\node[draw, circle, align=center] (node0) at (0,3) {
					$\boldsymbol{0}$ \textbf{(North)} \\
					$\demand[0] = 3000$
				};

			\node[draw, circle, align=center] (node1) at (0,0) {
					$\boldsymbol{1}$ \textbf{(South)} \\
					$\demand[1] = 1000$
				};

			\node[draw, circle, align=center] (node2) at (3,1.5) {
					$\boldsymbol{2}$ \textbf{(Center)} \\
					$\demand[2] = 6000$
				};

			\draw[->, thick] (node0.south) -- (node1.north)
				node[midway, left, align=center] {$\resistanceletter = 10^{-5}$};

			\draw[->, thick] (node0.east) -- (node2.north west)
				node[midway, above right, align=center] {$\resistanceletter = 10^{-5}$};

			\draw[->, thick] (node2.south west) -- (node1.east)
				node[midway, below right, align=center] {$\resistanceletter = 10^{-5}$};
		\end{tikzpicture}
	\end{minipage}%
	\hfill
	\begin{minipage}[c]{0.48\textwidth}
		\centering
		\scriptsize
		\begin{tabular}{|c|c|c|c|c|}
			\hline
			\textbf{Node} & \textbf{Tech} & \textbf{Cap.} & \textbf{Cost} & \textbf{Poll.} \\
			\hline
			\multirow{3}{*}{0 (N)} & Solar & 1800 & 0 & 0 \\
					       & Coal  & 1800 & 40 & 1 \\
					       & Gas   & 2400 & 80 & 0.5 \\
					       \hline
			\multirow{3}{*}{2 (C)} & Solar & 2400 & 0 & 0 \\
					       & Coal  & 1200 & 40 & 1 \\
					       & Gas   & 8400 & 80 & 0.5 \\
					       \hline
			\multirow{3}{*}{1 (S)} & Solar & 200  & 0 & 0 \\
					       & Coal  & 1800 & 40 & 1 \\
					       & Gas   & 1000 & 80 & 0.5 \\
					       \hline
		\end{tabular}
	\end{minipage}
	\caption{Three-node network with associated generation technologies and characteristics for the Chilean electricity market instance. In this network, the generation capacity of the center is not replaceable, as the south and the north cannot satisfy the network demand without it.}
	\label{fig:network-and-table}
\end{figure}

\subsection{Numerical sensitivity analysis} \label{sub:sensitivity-analysis}

We now analyze how market equilibria respond to variations in the ISO's pollution control measures specifically tightening the pollution limit \(\pollutionlimit\) or increasing the pollution cost \(\pollutioncost\). Since these measures are meant to be mutually exclusive in an applied setting, whenever we test the effect of pollution costs, the limit is set high enough to never be active during the experiment. Correspondingly, when the pollution limit is being tested, the costs \(\pollutioncost\) are all set to \(0\). The goal is to assess the effects of these mechanisms on flows, dispatch, emissions, and costs for the ISO and market agents. In reporting ISO costs, we focus on the \textit{effective market cost} under regulation, computed as
\[-\qty(\sum_{\node} \agentcost(\bidvector^*) - \sum_{\node, \tech} \truecost \production) = \sum_{\node, \tech} (\price - \pollutionfactor(\pollutioncost + \tax))\production,\]
which reflects the price paid adjusted by pollution penalties.

\subsubsection{Flow dynamics on instance \ref{fig:instance-1-asymmetric}}

Figure \ref{fig:pollution-cost-SA-instance-1} illustrates the effect of increasing the pollution cost from \(0\) to \(1\) on the equilibrium flow balance. As the cost rises, production gradually shifts from agent 0 to agent 1. Agent 0, initially a net exporter, becomes a net importer, while agent 1 takes over a larger share of the supply. This is accompanied by a rise in both agents' equilibrium bids, which influences their profits.

Recall that, using \ref{KKT-nodes} and the fact that the pollution constraint is inactive, the utilities of the agents in this one-node case reduce to the pay-as-bid formula: $u_\node(\bidletter_\node, \productionletter_\node) = (\bidletter_\node - \hat{\bidletter}_\node)\productionletter_\node$. As shown in Figure \ref{fig:pollution-cost-SA-instance-1} the profit of agent \(\node = 0\) remains largely unchanged: although its production decreases, the higher selling price compensates for the loss. Agent \(\node = 1\), on the other hand, benefits the most, producing more energy and selling it at a higher price. As a result, the ISO's total network operation cost increases.

The effect exhibits \textit{diminishing returns}: although the operation cost rises linearly, the rate of change in both energy mix and overall pollution reduction slows. This dynamic suggests limits to the efficiency of this cost-based pollution control mechanism.

\begin{figure}[htbp]
	\centering
	\begin{tikzpicture}
		\begin{groupplot}[
			group style={
				group size=3 by 2,
				horizontal sep=0.5cm,
				vertical sep=1.4cm,
			},
			title style={yshift=-1.3ex},
			xlabel style={yshift=1.2ex},
			width=5cm,
			height=3.6cm,
			xlabel={},
			grid=major,
			every axis plot/.append style={mark=*, mark size=1.5pt}
			]

			\pgfplotstableread[col sep = comma]{instance-1-pollution_cost-SA.csv}\instanceoneflow

			\nextgroupplot[
				title=Network energy mix,
				legend style={at={(0.5, -0.45)}, anchor=south, legend columns=2, font=\footnotesize},
				legend cell align={left},
				stack plots = y,
				ymin=0
			]
			\addplot[
				fill=RedOrange,
				mark options = {draw=black},
				]
				table [col sep=comma, x=pollution_cost, y=production_node0_tech0]
				{\instanceoneflow}\closedcycle;
			\addlegendentry{node 0}
			\addplot[
				fill=RoyalBlue,
				mark options = {draw=black},
				mark=square*
				]
				table [col sep=comma, x=pollution_cost, y=production_node1_tech0]
				{\instanceoneflow}\closedcycle;
			\addlegendentry{node 1}

			\nextgroupplot[
				title=Flow \(0 \to 1\),
				only marks
			]
			\addplot[
				color=Black,
				mark options = {draw=black}
				]
				table [col sep=comma, x=pollution_cost, y=flow_0to1]
				{\instanceoneflow};

			\nextgroupplot[
				title=Pollution by node,
				legend style={at={(0.5, -0.45)}, anchor=south, legend columns=2, font=\footnotesize},
				legend cell align={left},
				stack plots = y,
				ymin=0
			]
			\addplot[
				fill=RedOrange,
				mark options = {draw=black},
				]
				table [
				col sep=comma,
				x=pollution_cost,
				y=pollution0
				]
				{\instanceoneflow}\closedcycle;
			\addlegendentry{node 0}
			\addplot[
				fill=RoyalBlue,
				mark options = {draw=black},
				mark=square*
				]
				table [
				col sep=comma,
				x=pollution_cost,
				y=pollution1
				]
				{\instanceoneflow}\closedcycle;
			\addlegendentry{node 1}

			\nextgroupplot[
				title=Agent utilities,
				only marks,
				legend style={
					at={(0.5, -0.45)},
					anchor=south, 
					legend columns=2,
					font=\footnotesize
				},
				legend cell align={left},
				ymin=0
			]
			\addplot[
				color=RedOrange,
				mark options = {draw=black}
				]
				table [col sep=comma, x=pollution_cost, y=agent_utility0]
				{\instanceoneflow};
			\addlegendentry{node 0}
			\addplot[
				color=RoyalBlue,
				mark options = {draw=black},
				mark = square*
				]
				table [col sep=comma, x=pollution_cost, y=agent_utility1]
				{\instanceoneflow};
			\addlegendentry{node 1}

			\nextgroupplot[
				title=Bids,
				only marks,
				legend style={
					at={(0.5, -0.45)},
					anchor=south, 
					legend columns=2,
					font=\footnotesize
				},
				legend cell align={left},
			]
			\addplot[
				color=RedOrange,
				mark options = {draw=black}
				]
				table [col sep=comma, x=pollution_cost, y=bid_node0_tech0]
				{\instanceoneflow};
			\addlegendentry{node 0}
			\addplot[
				color=RoyalBlue,
				mark options = {draw=black},
				mark = square*
				]
				table [col sep=comma, x=pollution_cost, y=bid_node1_tech0]{\instanceoneflow};
			\addlegendentry{node 1}

			\nextgroupplot[
				title=Operation cost,
				only marks,
				ymin=0
			]
			\addplot[
				color=Black,
				mark options = {draw=black}
				]
				table [col sep=comma, x=pollution_cost, y=operation_cost]{\instanceoneflow};

		\end{groupplot}
	\end{tikzpicture}
	\caption{Pollution cost sensitivity analysis on Instance \ref{fig:instance-1-asymmetric}. In all plots, the \(x\) axis represents pollution cost and dots depict corresponding variables at a computed market equilibrium. Here and in the plots that follow, energy mix and pollution plots are stacked.}
	\label{fig:pollution-cost-SA-instance-1}
\end{figure}
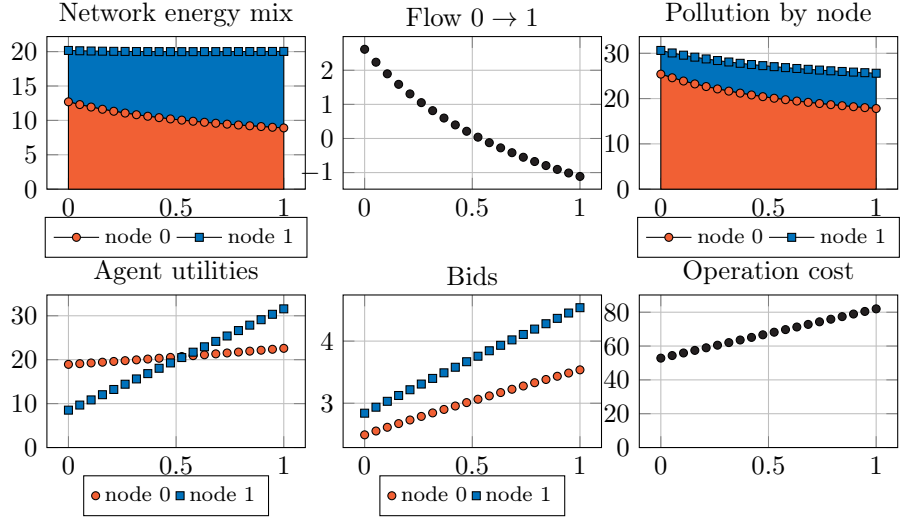

The remaining experiments are based on instances where the available technologies are symmetrically distributed across nodes. That is, each node is equipped with the same set of production technologies -—e.g., gas, solar, wind, or oil—- with identical or nearly identical cost and pollution parameters. This is a realistic assumption for national-scale models, where each node represents a region rather than an individual plant. While the current computational framework cannot yet scale to more granular representations (e.g., city-level or plant-level modeling), this experiment serves as a small-scale demonstration of dynamics that would emerge in such settings if regional technological asymmetries were introduced.

\subsubsection{Pollution cost on instance \ref{fig:instance-2}}

In Figure \ref{fig:pollution-cost-SA-instance-2}, we raise the cost associated with the most polluting technology in Instance \ref{fig:instance-2} from 0 to 0.4 and observe the market's reaction.

\begin{figure}[htbp]
	\centering
	\begin{tikzpicture}
		\begin{groupplot}[
			group style={
				group size=3 by 2,
				horizontal sep=0.74cm,
				vertical sep=1.45cm,
			},
			title style={yshift=-1.3ex},
			xlabel style={yshift=1.2ex},
			width=4.6cm,
			height=3.6cm,
			xlabel={},
			grid=major,
			every axis plot/.append style={mark=*, mark size=1.5pt}
			]

			\pgfplotstableread[col sep = comma]{instance-2-pollution_cost-SA.csv}\instancetwocost

			\nextgroupplot[
				title=Energy mix agent 0,
				legend style={at={(0.5, -0.45)}, anchor=south, legend columns=2, font=\footnotesize},
				legend cell align={left},
				stack plots=y
				]
			\addplot[
				fill=SkyBlue,
				mark options = {draw=black},
				]
				table [col sep=comma, x=pollution_cost, y=production_node0_tech0]
				{\instancetwocost}\closedcycle;
			\addlegendentry{tech 0}
			\addplot[
				fill=Maroon,
				mark options = {draw=black},
				mark = square*
				]
				table [col sep=comma, x=pollution_cost, y=production_node0_tech1]
				{\instancetwocost}\closedcycle;
			\addlegendentry{tech 1}

			\nextgroupplot[
				title=Energy mix agent 1,
				legend style={at={(0.5, -0.45)}, anchor=south, legend columns=2, font=\footnotesize},
				legend cell align={left},
				stack plots = y
				]
			\addplot[
				fill=SkyBlue,
				mark options = {draw=black},
				]
				table [col sep=comma, x=pollution_cost, y=production_node1_tech0]
				{\instancetwocost}\closedcycle;
			\addlegendentry{tech 0}
			\addplot[
				fill=Maroon,
				mark options = {draw=black},
				mark = square*
				]
				table [col sep=comma, x=pollution_cost, y=production_node1_tech1]
				{\instancetwocost}\closedcycle;
			\addlegendentry{tech 1}

			\nextgroupplot[
				title=Flow \(0 \rightarrow 1\),
				only marks,
				ymin=-1,
				ymax=1
			]
			\addplot[color=Black]
				table [
				col sep=comma,
				x=pollution_cost,
				y={flow_0to1},
				]
				{\instancetwocost};

			\nextgroupplot[
				title=Pollution by node,
				legend style={
					at={(0.5, -0.45)},
					anchor=south,
					font=\footnotesize,
					legend columns=2
				},
				legend cell align={left},
				stack plots = y,
				ymin=0
				]
			\addplot[
				fill=RoyalBlue,
				mark options = {draw=black},
				]
				table [col sep=comma, x=pollution_cost, y=pollution0]
				{\instancetwocost}\closedcycle;
			\addlegendentry{node 0}
			\addplot[
				fill=RedOrange,
				mark options = {draw=black},
				mark = square*
				]
				table [col sep=comma, x=pollution_cost, y=pollution1]
				{\instancetwocost}\closedcycle;
			\addlegendentry{node 1}

			\nextgroupplot[
				title=Agent utilities,
				only marks,
				legend style={
					at={(0.5, -0.45)},
					anchor=south,
					font=\footnotesize,
					legend columns=2
				},
				legend cell align={left},
				ymin=0
				]
			\addplot[mark=x, color=RoyalBlue]
				table [
				col sep=comma,
				x=pollution_cost,
				y=agent_utility0
				]
				{\instancetwocost};
			\addlegendentry{node 0}
			\addplot[mark=+, color=RedOrange]
				table [
				col sep=comma,
				x=pollution_cost,
				y=agent_utility1
				]
				{\instancetwocost};
			\addlegendentry{node 1}

			\nextgroupplot[
			title=Operation cost,
			only marks,
			ymin=0
			]
			\addplot[color=Black]
				table [
				col sep=comma,
				x=pollution_cost,
				y=operation_cost
				]
				{\instancetwocost};

		\end{groupplot}
	\end{tikzpicture}
	\caption{Pollution cost sensitivity analysis on Instance \ref{fig:instance-2}. In all plots, the \(x\) axis represents the pollution cost \(\pollutioncost[1]\) associated to \(\tech = 1\).}
	\label{fig:pollution-cost-SA-instance-2}
\end{figure}
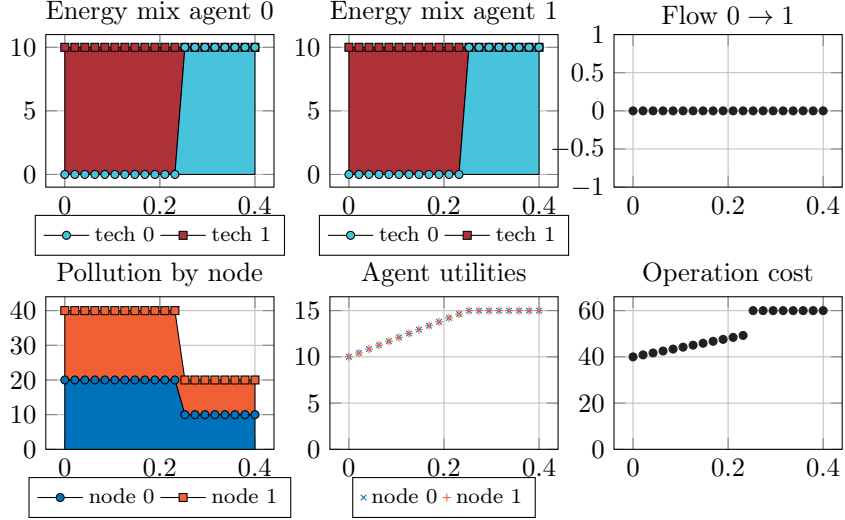

The energy flow between the nodes is equal to zero indicating that across all equilibria, each agent produces locally to fully satisfy its own demand (10 units), resulting in no transmission between nodes. However, due to the presence of multiple technologies, the internal transition between them, occurring at around \(\pollutioncost[1] = 0.25\), is abrupt rather than gradual. This is expected as there is no \textit{friction} or gradual substitution cost between different technologies within the same agent, in contrast to inter-agent interactions.

Interestingly, even though the production mix jumps on a clearly defined breakpoint, the agent costs remain continuous functions of \(\pollutioncost[1]\). Utilities, driven by nodal prices, increase linearly with pollution cost, leading to higher profitability for the participating agents, yet the production portfolio remains unchanged until prices reach a threshold that triggers a full technology switch. A potential direction for future research would be to formally characterize the conditions under which this continuity property holds. As such insights could allow analytical computation of these breakpoints on more complex market instances.

\subsubsection{Pollution limit on instance \ref{fig:instance-2}}

We now evaluate, still on Instance \ref{fig:instance-2}, the effect of the second available regulatory mechanism: the pollution limit. In Figure \ref{fig:pollution-limit-SA-instance-2-symmetric}, the pollution limit \(\pollutionlimit\) is gradually decreased from its initial value of 42, which imposes no binding constraint on the equilibrium, down to more restrictive levels that compel the substitution of polluting technologies with cleaner ones. While variability can be observed in the agents' energy mix plots, the overall trend is clear: the use of technology \(\tech = 1\) decreases in favor of the cleaner alternative \(\tech = 0\) as the constraint tightens. As previously mentioned, this variability stems from the symmetry of the network, which allows non-uniqueness of the ISO solution for the equilibrium bids.

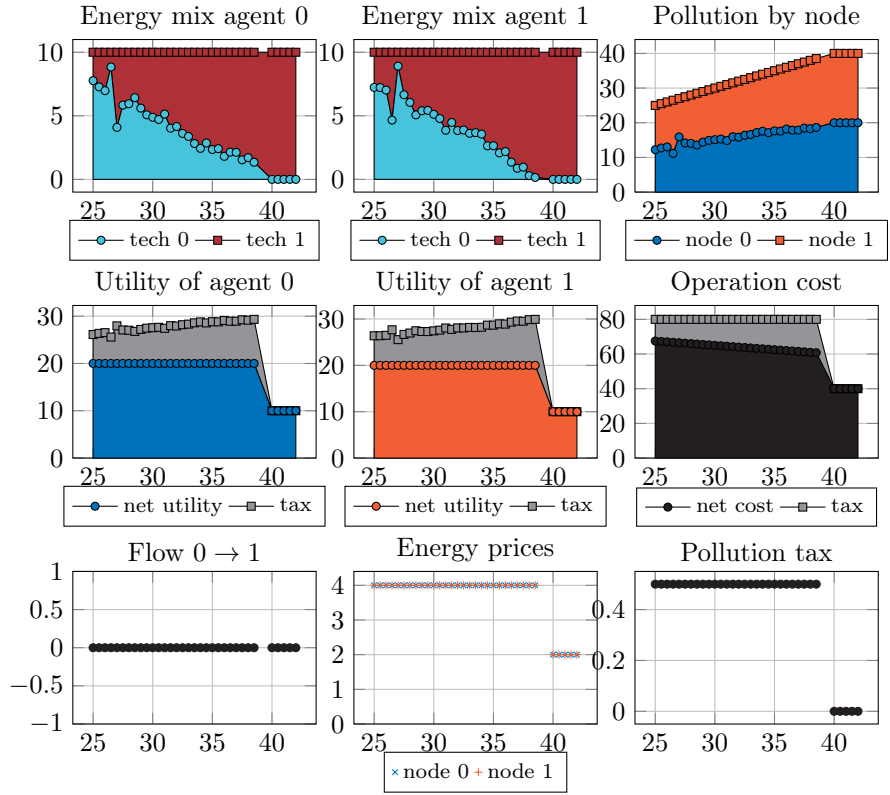
\begin{figure}[htbp]
	\centering
	\begin{tikzpicture}
		\begin{groupplot}[
			group style={
				group size=3 by 3,
				horizontal sep=0.5cm,
				vertical sep=1.5cm,
			},
			title style={yshift=-1.3ex}, xlabel style={yshift=1.2ex},
			width=4.8cm,
			height=3.6cm,
			xlabel={},
			grid=major,
			every axis plot/.append style={mark=*, mark size=1.5pt}
			]

			\pgfplotstableread[col sep = comma]{instance-2-pollution_limit-SA-symmetric.csv}\instancetwosymmetric

			\nextgroupplot[
				title=Energy mix agent 0,
				legend style={
					at={(0.5, -0.45)},
					anchor=south,
					font=\footnotesize,
					legend columns=2
				},
				legend cell align={left},
				stack plots = y
				]
			\addplot[
				fill=SkyBlue,
				mark options = {draw=black}
				]
				table [col sep=comma, x=pollution_limit, y=production_node0_tech0]
				{\instancetwosymmetric}\closedcycle;
			\addlegendentry{tech 0}
			\addplot[
				fill=Maroon,
				mark options = {draw=black},
				mark = square*
				]
				table [col sep=comma, x=pollution_limit, y=production_node0_tech1]
				{\instancetwosymmetric}\closedcycle;
			\addlegendentry{tech 1}

			\nextgroupplot[
				title=Energy mix agent 1,
				legend style={
					at={(0.5, -0.45)},
					anchor=south,
					font=\footnotesize,
					legend columns=2
				},
				legend cell align={left},
				stack plots = y
				]
			\addplot[
				fill=SkyBlue,
				mark options = {draw=black}
				]
				table [col sep=comma, x=pollution_limit, y=production_node1_tech0]
				{\instancetwosymmetric}\closedcycle;
			\addlegendentry{tech 0}
			\addplot[
				fill=Maroon,
				mark options = {draw=black},
				mark = square*
				]
				table [col sep=comma, x=pollution_limit, y=production_node1_tech1]
				{\instancetwosymmetric}\closedcycle;
			\addlegendentry{tech 1}

			\nextgroupplot[
				title=Pollution by node,
				legend style={
					at={(0.5, -0.45)},
					anchor=south,
					font=\footnotesize,
					legend columns=2
				},
				legend cell align={left},
				stack plots = y,
				ymin=0
				]
			\addplot[
				fill=RoyalBlue,
				mark options = {draw=black},
				]
				table [col sep=comma, x=pollution_limit, y=pollution0]
				{\instancetwosymmetric}\closedcycle;
			\addlegendentry{node 0}
			\addplot[
				fill=RedOrange,
				mark options = {draw=black},
				mark = square*
				]
				table [col sep=comma, x=pollution_limit, y=pollution1]
				{\instancetwosymmetric}\closedcycle;
			\addlegendentry{node 1}

			\nextgroupplot[
				title=Utility of agent 0,
				legend style={
					at={(0.5, -0.45)},
					anchor=south,
					font=\footnotesize,
					legend columns=2
				},
				legend cell align={left},
				stack plots = y,
				ymin=0
				]
			\addplot[
				fill=RoyalBlue,
				mark options = {draw=black},
				]
				table [col sep=comma, x=pollution_limit, y=agent_utility0]
				{\instancetwosymmetric}\closedcycle;
			\addlegendentry{net utility}
			\addplot[
				fill=Gray,
				mark options = {draw=black},
				mark = square*
				]
				table [col sep=comma, x=pollution_limit, y=agent_tax0]
				{\instancetwosymmetric}\closedcycle;
			\addlegendentry{tax}

			\nextgroupplot[
				title=Utility of agent 1,
				legend style={
					at={(0.5, -0.45)},
					anchor=south,
					font=\footnotesize,
					legend columns=2
				},
				legend cell align={left},
				stack plots = y,
				ymin=0
				]
			\addplot[
				fill=RedOrange,
				mark options = {draw=black},
				]
				table [col sep=comma, x=pollution_limit, y=agent_utility1]
				{\instancetwosymmetric}\closedcycle;
			\addlegendentry{net utility}
			\addplot[
				fill=Gray,
				mark options = {draw=black},
				mark = square*
				]
				table [col sep=comma, x=pollution_limit, y=agent_tax1]
				{\instancetwosymmetric}\closedcycle;
			\addlegendentry{tax}

			\nextgroupplot[
				title=Operation cost,
				legend style={
					at={(0.5, -0.45)},
					anchor=south,
					font=\footnotesize,
					legend columns=2
				},
				legend cell align={left},
				stack plots = y,
				ymin=0
				]
			\addplot[
				fill=Black,
				mark options = {draw=black},
				]
				table [col sep=comma, x=pollution_limit, y=operation_cost]
				{\instancetwosymmetric}\closedcycle;
			\addlegendentry{net cost}
			\addplot[
				fill=Gray,
				mark options = {draw=black},
				mark = square*
				]
				table [
				col sep=comma,
				x=pollution_limit,
				y=total_tax
				]
				{\instancetwosymmetric}\closedcycle;
			\addlegendentry{tax}

			\nextgroupplot[
				title=Flow \(0 \rightarrow 1\),
				only marks,
				ymin=-1,
				ymax=1
				]
			\addplot[color=Black]
				table [
				col sep=comma,
				x=pollution_limit,
				y=flow_0to1
				]
				{\instancetwosymmetric};

			\nextgroupplot[
				title=Energy prices,
				only marks,
				legend style={
					at={(0.5, -0.45)},
					anchor=south,
					legend columns=2,
					font=\footnotesize
				},
				legend cell align={left},
				ymin=0
				]
			\addplot[mark=x, color=RoyalBlue]
				table [col sep=comma, x=pollution_limit, y=price0]
				{\instancetwosymmetric};
			\addlegendentry{node 0}
			\addplot[mark=+, color=RedOrange]
				table [col sep=comma, x=pollution_limit, y=price1]
				{\instancetwosymmetric};
			\addlegendentry{node 1}

			\nextgroupplot[
			title=Pollution tax,
			only marks
			]
			\addplot[color=Black]
				table [
				col sep=comma,
				x=pollution_limit,
				y=tax
				]
				{\instancetwosymmetric};

		\end{groupplot}
	\end{tikzpicture}
	\caption{Pollution limit sensitivity analysis on Instance \ref{fig:instance-2}. In all plots, the \(x\) axis represents the pollution limit \(\pollutionlimit\). Missing points correspond to unsuccessful grid search runs.}
	\label{fig:pollution-limit-SA-instance-2-symmetric}
\end{figure}

Similarly to the previous cases, the introduction of a binding pollution limit benefits the agents by increasing their profit margins. The shift in the energy mix, potentially disadvantageous due to changes in production costs, is offset by a reduction in the pollution tax burden. This tax, computed as \(\tax \cdot \sum_{\tech} \pollutionfactor \production\), represents the total penalty associated with emissions at equilibrium. The shaded gray region over the agent utility depicted in Figure \ref{fig:pollution-limit-SA-instance-2-symmetric}, illustrates the utility that would have been determined in the absence of pollution taxation.

An important consequence of this regulatory mechanism is its disruptive effect on the continuity of prices, which remained intact under the pollution cost mechanism. As shown in the price and agent cost plots, prices exhibit a discontinuous jump precisely at the point where the pollution constraint becomes active. After this jump, prices remain flat, regardless of further tightening of the limit. Meanwhile, the ISO's monetary cost continues to rise as the pollution limit is lowered. This behavior can be explained by the decrease in the total collected tax $\tax \cdot \sum_{\node, \tech} \pollutionfactor \production$, which is plotted in gray over the ISO's net cost. As network pollution declines, the total tax diminishes. Leading to a greater net cost for the ISO to operate the market.

Repeating the same experiment on the asymmetrical Instance \ref{fig:instance-2-asymmetric} leads to the results shown in Figure \ref{fig:pollution-limit-SA-instance-2-asymmetric}. The ISO dispatch solution at the computed equilibria is unique. However, we observe small variations in price-related quantities such as nodal prices and agent costs. These variations become more pronounced when the asymmetry is intensified or when nonzero pollution costs are introduced on top of an already active pollution limit.

This variability is linked to the existence of multiple equilibria for each value of \(\pollutionlimit\), as the grid search routine is configured to stop after finding the first equilibrium. When this restriction is relaxed, multiple verified equilibria can indeed be found, leading to the appearance of intervals of possible equilibrium values for nodal prices and other quantities under a fixed pollution constraint.
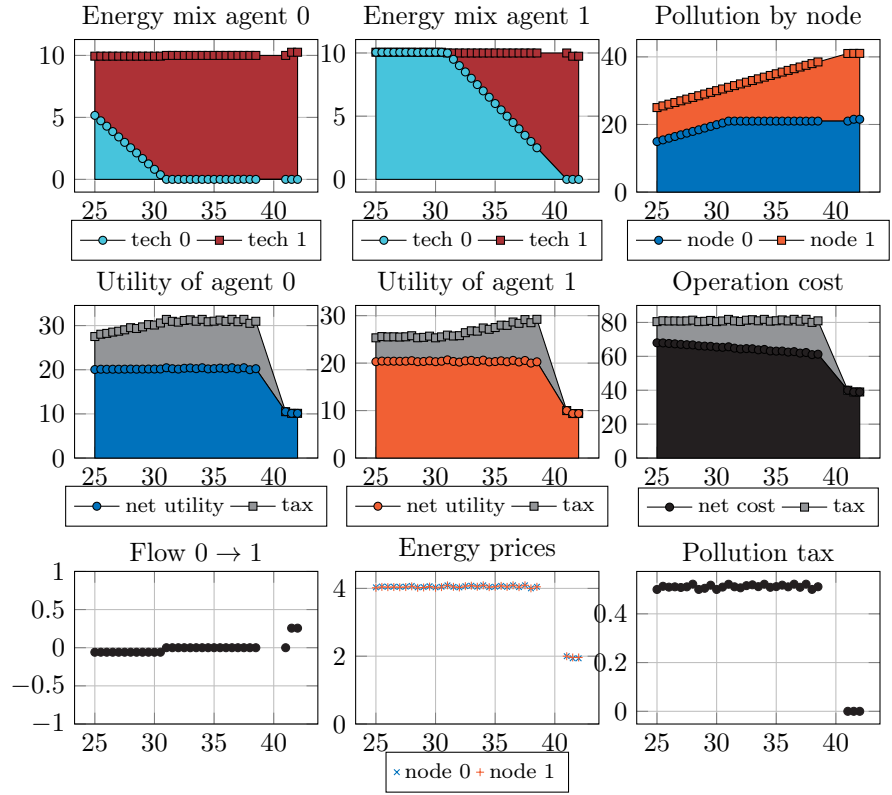
\begin{figure}[htbp]
	\begin{tikzpicture}
		\begin{groupplot}[
			group style={
				group size=3 by 3,
				horizontal sep=0.5cm,
				vertical sep=1.5cm,
			},
			title style={yshift=-1.3ex}, xlabel style={yshift=1.2ex},
			width=4.8cm,
			height=3.6cm,
			xlabel={},
			grid=major,
			every axis plot/.append style={mark=*, mark size=1.5pt}
			]

			\pgfplotstableread[col sep = comma]{instance-2-pollution_limit-SA-asymmetric.csv}\instancetwoasymmetric

			\nextgroupplot[
				title=Energy mix agent 0,
				legend style={
					at={(0.5, -0.45)},
					anchor=south,
					font=\footnotesize,
					legend columns=2
				},
				legend cell align={left},
				stack plots = y
				]
			\addplot[
				fill=SkyBlue,
				mark options = {draw=black}
				]
				table [col sep=comma, x=pollution_limit, y=production_node0_tech0]
				{\instancetwoasymmetric}\closedcycle;
			\addlegendentry{tech 0}
			\addplot[
				fill=Maroon,
				mark options = {draw=black},
				mark = square*
				]
				table [col sep=comma, x=pollution_limit, y=production_node0_tech1]
				{\instancetwoasymmetric}\closedcycle;
			\addlegendentry{tech 1}

			\nextgroupplot[
				title=Energy mix agent 1,
				legend style={
					at={(0.5, -0.45)},
					anchor=south,
					font=\footnotesize,
					legend columns=2
				},
				legend cell align={left},
				stack plots = y
				]
			\addplot[
				fill=SkyBlue,
				mark options = {draw=black}
				]
				table [col sep=comma, x=pollution_limit, y=production_node1_tech0]
				{\instancetwoasymmetric}\closedcycle;
			\addlegendentry{tech 0}
			\addplot[
				fill=Maroon,
				mark options = {draw=black},
				mark = square*
				]
				table [col sep=comma, x=pollution_limit, y=production_node1_tech1]
				{\instancetwoasymmetric}\closedcycle;
			\addlegendentry{tech 1}

			\nextgroupplot[
				title=Pollution by node,
				legend style={
					at={(0.5, -0.45)},
					anchor=south,
					font=\footnotesize,
					legend columns=2
				},
				legend cell align={left},
				stack plots = y,
				ymin=0
				]
			\addplot[
				fill=RoyalBlue,
				mark options = {draw=black},
				]
				table [col sep=comma, x=pollution_limit, y=pollution0]
				{\instancetwoasymmetric}\closedcycle;
			\addlegendentry{node 0}
			\addplot[
				fill=RedOrange,
				mark options = {draw=black},
				mark = square*
				]
				table [col sep=comma, x=pollution_limit, y=pollution1]
				{\instancetwoasymmetric}\closedcycle;
			\addlegendentry{node 1}

			\nextgroupplot[
				title=Utility of agent 0,
				legend style={
					at={(0.5, -0.45)},
					anchor=south,
					font=\footnotesize,
					legend columns=2
				},
				legend cell align={left},
				stack plots = y,
				ymin=0
				]
			\addplot[
				fill=RoyalBlue,
				mark options = {draw=black},
				]
				table [col sep=comma, x=pollution_limit, y=agent_utility0]
				{\instancetwoasymmetric}\closedcycle;
			\addlegendentry{net utility}
			\addplot[
				fill=Gray,
				mark options = {draw=black},
				mark = square*
				]
				table [col sep=comma, x=pollution_limit, y=agent_tax0]
				{\instancetwoasymmetric}\closedcycle;
			\addlegendentry{tax}

			\nextgroupplot[
				title=Utility of agent 1,
				legend style={
					at={(0.5, -0.45)},
					anchor=south,
					font=\footnotesize,
					legend columns=2
				},
				legend cell align={left},
				stack plots = y,
				ymin=0
				]
			\addplot[
				fill=RedOrange,
				mark options = {draw=black},
				]
				table [col sep=comma, x=pollution_limit, y=agent_utility1]
				{\instancetwoasymmetric}\closedcycle;
			\addlegendentry{net utility}
			\addplot[
				fill=Gray,
				mark options = {draw=black},
				mark = square*
				]
				table [col sep=comma, x=pollution_limit, y=agent_tax1]
				{\instancetwoasymmetric}\closedcycle;
			\addlegendentry{tax}

			\nextgroupplot[
				title=Operation cost,
				legend style={
					at={(0.5, -0.45)},
					anchor=south,
					font=\footnotesize,
					legend columns=2
				},
				legend cell align={left},
				stack plots = y,
				ymin=0
				]
			\addplot[
				fill=Black,
				mark options = {draw=black},
				]
				table [col sep=comma, x=pollution_limit, y=operation_cost]
				{\instancetwoasymmetric}\closedcycle;
			\addlegendentry{net cost}
			\addplot[
				fill=Gray,
				mark options = {draw=black},
				mark = square*
				]
				table [
				col sep=comma,
				x=pollution_limit,
				y=total_tax
				]
				{\instancetwoasymmetric}\closedcycle;
			\addlegendentry{tax}

			\nextgroupplot[
				title=Flow \(0 \rightarrow 1\),
				only marks,
				ymin=-1,
				ymax=1
				]
			\addplot[color=Black]
				table [
				col sep=comma,
				x=pollution_limit,
				y=flow_0to1
				]
				{\instancetwoasymmetric};

			\nextgroupplot[
				title=Energy prices,
				only marks,
				legend style={
					at={(0.5, -0.45)},
					anchor=south,
					legend columns=2,
					font=\footnotesize
				},
				legend cell align={left},
				ymin=0
				]
			\addplot[mark=x, color=RoyalBlue]
				table [col sep=comma, x=pollution_limit, y=price0]
				{\instancetwoasymmetric};
			\addlegendentry{node 0}
			\addplot[mark=+, color=RedOrange]
				table [col sep=comma, x=pollution_limit, y=price1]
				{\instancetwoasymmetric};
			\addlegendentry{node 1}

			\nextgroupplot[
			title=Pollution tax,
			only marks
			]
			\addplot[color=Black]
				table [
				col sep=comma,
				x=pollution_limit,
				y=tax
				]
				{\instancetwoasymmetric};

		\end{groupplot}
	\end{tikzpicture}
	\caption{Pollution limit sensitivity analysis on asymmetric version of Instance \ref{fig:instance-2}. In all plots, the \(x\) axes represents the pollution limit \(\pollutionlimit\). Missing points correspond to unsuccessful grid search runs.}
	\label{fig:pollution-limit-SA-instance-2-asymmetric}
\end{figure}

The effect of the pollution limit mechanism on the energy mix can be qualitatively described as follows: the technology \(\tech = 1\) of agent \(\node = 1\) is the first to be displaced from the market in favor of its alternative \(\tech = 0\), while the same technology for \(\node = 1\) remains active until \(\production[1][1] = 0\). This is consistent with Proposition \ref{prop:technology-activation-active-pollution}, which implies that for any equilibrium associated to a unique ISO solution (unlike the multiplicity observed in Figure \ref{fig:pollution-limit-SA-instance-2-symmetric}), it is infeasible for more than one node to have two partially dispatched technologies.

Missing points in both plots correspond to values of the pollution limit for which no equilibrium could be determined by the grid search routine. This behavior persisted even when the grid was refined, suggesting that certain pollution limits—especially those that are tight yet not prohibitively so—may cause equilibria to become nonexistent or exceedingly difficult to locate numerically.

\subsubsection{Pollution cost on instance \ref{fig:network-and-table}}

Instance \ref{fig:network-and-table} is the most computationally challenging case studied in this work. Preliminary testing revealed that equilibria associated with active pollution limit values could not be reliably computed. The sensitivity analysis for this instance thus focuses solely on the effects of pollution cost, seen in Figure \ref{fig:pollution-cost-SA-instance-4}. The grid search algorithm struggles to find valid equilibria near a critical value of \(\pollutioncost[1]\) which triggers a switch in the market's overall technology mix. This difficulty is due to a combination between not finding many candidate MCP solutions and a higher proportion of them being discarded by the agent optimality test as not being true Nash equilibria. Solar energy (technology 0) is always dispatched at full capacity, as it \textit{Pareto-dominates} both coal (technology 1) and gas (technology 2), being both cheaper and cleaner. At low pollution cost values, gas is generally unused (except at node \(\node = 2\) where other resources are exhausted) and coal is preferred due to its lower cost.

As the pollution cost on coal increases, the dispatch pattern shifts sharply, mimicking the dynamics seen in Figure \ref{fig:pollution-cost-SA-instance-2}. Coal is replaced by gas, resulting in a significant reduction in network-wide emissions. The flow pattern also reverses: the gas-abundant northern region becomes a net exporter to the center, overtaking the previously dominant coal-heavy southern region.

At the market level, the effect previously seen in Figure \ref{fig:pollution-cost-SA-instance-2} is notably inverted: Agent profitability is reduced as the pollution cost increases to the critical value, and after, instead of them being constant, agent profits increase largely driven by rising nodal prices. This is a consequence of the previously mentioned network's lack of \(N - 1\) robustness, which confers market power to the agents. A \textit{desirable} level of pollution cost is thus suggested to be one that achieves the technology switch without pushing prices prohibitively high. Further insight into this critical regime would require either significantly more computational time or a different numerical approach.

\begin{figure}[htbp]
	\centering
	\begin{tikzpicture}
		\begin{groupplot}[
			group style={
				group size=2 by 4,
				horizontal sep=1.2cm,
				vertical sep=1.7cm,
			},
			width=6.2cm,
			height=3.8cm,
			xlabel={},
			grid=major,
			every axis plot/.append style={
				mark=*,
				mark size=1.5pt
			}
			]

			\pgfplotstableread[col sep = comma]{instance-4-pollution_cost-SA-mod.csv}\instancefourcost

			\nextgroupplot[
				title=Energy mix agent 0,
				legend style={
					at={(0.5, -0.36)},
					anchor=south,
					font=\tiny,
					legend columns=3
				},
				legend cell align={left},
				stack plots = y,
				ymin=0
				]
			\addplot[
				fill=CornflowerBlue,
				mark options = {draw=black}
				]
				table [col sep=comma, x=pollution_cost, y=production_node0_tech0]
				{\instancefourcost}\closedcycle;
			\addlegendentry{tech 0}
			\addplot[
				mark = square*,
				fill = Maroon,
				mark options = {draw=black}
				]
				table [col sep=comma, x=pollution_cost, y=production_node0_tech1]
				{\instancefourcost}\closedcycle;
			\addlegendentry{tech 1}
			\addplot[
				mark = triangle*,
				fill=Dandelion,
				mark options = {draw=black}
				]
				table [col sep=comma, x=pollution_cost, y=production_node0_tech2]
				{\instancefourcost}\closedcycle;
			\addlegendentry{tech 2}

			\nextgroupplot[
				title=Energy mix agent 1,
				legend style={at={(0.5, -0.36)},
					anchor=south,
					font=\tiny,
					legend columns=3
				},
				legend cell align={left},
				stack plots = y,
				ymin=0
				]
			\addplot[
				fill=CornflowerBlue,
				mark options = {draw=black}
				]
				table [col sep=comma, x=pollution_cost, y=production_node1_tech0]
				{\instancefourcost}\closedcycle;
			\addlegendentry{tech 0}
			\addplot[
				mark = square*,
				fill = Maroon,
				mark options = {draw=black}
				]
				table [col sep=comma, x=pollution_cost, y=production_node1_tech1]
				{\instancefourcost}\closedcycle;
			\addlegendentry{tech 1}
			\addplot[
				mark = triangle*,
				fill=Dandelion,
				mark options = {draw=black}
				]
				table [col sep=comma, x=pollution_cost, y=production_node1_tech2]
				{\instancefourcost}\closedcycle;
			\addlegendentry{tech 2}

			\nextgroupplot[
				title=Energy mix agent 2,
				legend style={
					at={(0.5, -0.36)},
					anchor=south,
					font=\tiny,
					legend columns=3
				},
				legend cell align={left},
				stack plots = y,
				ymin=0
				]
			\addplot[
				fill=CornflowerBlue,
				mark options = {draw=black}
				]
				table [col sep=comma, x=pollution_cost, y=production_node2_tech0]
				{\instancefourcost}\closedcycle;
			\addlegendentry{tech 0}
			\addplot[
				mark = square*,
				fill = Maroon,
				mark options = {draw=black}
				]
				table [col sep=comma, x=pollution_cost, y=production_node2_tech1]
				{\instancefourcost}\closedcycle;
			\addlegendentry{tech 1}
			\addplot[
				mark = triangle*,
				fill=Dandelion,
				mark options = {draw=black}
				]
				table [col sep=comma, x=pollution_cost, y=production_node2_tech2]
				{\instancefourcost}\closedcycle;
			\addlegendentry{tech 2}

			\nextgroupplot[
				title=Flows,
				only marks,
				legend style={
					at={(0.5, -0.36)},
					anchor=south,
					font=\tiny,
					legend columns=3
				},
				legend cell align={left}
				]
			\addplot[
				color=Periwinkle,
				mark options = {draw=black}
				]
				table [col sep=comma, x=pollution_cost, y=flow_0to1]
				{\instancefourcost};
			\addlegendentry{$0 \rightarrow 1$}
			\addplot[
				color=Sepia,
				mark options = {draw=black},
				mark=square*
				]
				table [col sep=comma, x=pollution_cost, y=flow_0to2]
				{\instancefourcost};
			\addlegendentry{$0 \rightarrow 2$}
			\addplot[
				color=Lavender,
				mark options = {draw=black},
				mark=triangle*
				]
				table [col sep=comma, x=pollution_cost, y=flow_2to1]
				{\instancefourcost};
			\addlegendentry{$2 \rightarrow 1$}

			\nextgroupplot[
				title=Energy prices,
				only marks,
				legend style={
					at={(0.5, -0.36)},
					anchor=south,
					font=\tiny,
					legend columns=3
				},
				legend cell align={left},
				ymin=0
				]
			\addplot[
				color=RoyalBlue,
				mark options = {draw=black}
				]
				table [col sep=comma, x=pollution_cost, y=price0]
				{\instancefourcost};
			\addlegendentry{node 0}
			\addplot[
				color=BurntOrange,
				mark options = {draw=black},
				mark = square*
				]
				table [col sep=comma, x=pollution_cost, y=price1]
				{\instancefourcost};
			\addlegendentry{node 1}
			\addplot[
				color=ForestGreen,
				mark options = {draw=black},
				mark = triangle*
				]
				table [col sep=comma, x=pollution_cost, y=price2]
				{\instancefourcost};
			\addlegendentry{node 2}

			\nextgroupplot[
				title=Pollution by node,
				legend style={
					at={(0.5, -0.36)},
					anchor=south,
					font=\tiny,
					legend columns=3
				},
				legend cell align={left},
				stack plots = y,
				ymin=0
				]
			\addplot[
				fill=RoyalBlue,
				mark options = {draw=black}
				]
				table [col sep=comma, x=pollution_cost, y=pollution0]
				{\instancefourcost}\closedcycle;
			\addlegendentry{node 0}
			\addplot[
				mark = square*,
				fill = BurntOrange,
				mark options = {draw=black}
				]
				table [col sep=comma, x=pollution_cost, y=pollution1]
				{\instancefourcost}\closedcycle;
			\addlegendentry{node 1}
			\addplot[
				mark = triangle*,
				fill=ForestGreen,
				mark options = {draw=black}
				]
				table [col sep=comma, x=pollution_cost, y=pollution2]
				{\instancefourcost}\closedcycle;
			\addlegendentry{node 2}

			\nextgroupplot[
				title=Agent utilities,
				only marks,
				legend style={
					at={(0.5, -0.36)},
					anchor=south,
					font=\tiny,
					legend columns=3
				},
				legend cell align={left},
				]
			\addplot[
				color=RoyalBlue,
				mark options = {draw=black}
				]
				table [
				col sep=comma,
				x=pollution_cost,
				y=agent_utility0
				]
				{\instancefourcost};
			\addlegendentry{node 0}
			\addplot[
				color=BurntOrange,
				mark options = {draw=black},
				mark = square*
				]
				table [
				col sep=comma,
				x=pollution_cost,
				y=agent_utility1
				]
				{\instancefourcost};
			\addlegendentry{node 1}
			\addplot[
				color=ForestGreen,
				mark options = {draw=black},
				mark = triangle*
				]
				table [
				col sep=comma,
				x=pollution_cost,
				y=agent_utility2
				]
				{\instancefourcost};
			\addlegendentry{node 2}

			\nextgroupplot[
				title=Operation cost,
				only marks,
				ymin=0
				]
			\addplot[color=Black]
				table [
				col sep=comma,
				x=pollution_cost,
				y=operation_cost
				]
				{\instancefourcost};

		\end{groupplot}
	\end{tikzpicture}
	\caption{Pollution cost sensitivity analysis on instance \ref{fig:network-and-table}. In all plots, the \(x\) axis represents the pollution cost \(\pollutioncost[1]\) associated to coal. Missing points represent unsuccessful grid search runs. Linear interpolation in the energy mix and pollution plots was kept to illustrate the emerging transition but real values at equilibrium (if said equilibrium exists) are not presumed to fit the given interpolation.}
	\label{fig:pollution-cost-SA-instance-4}
\end{figure}
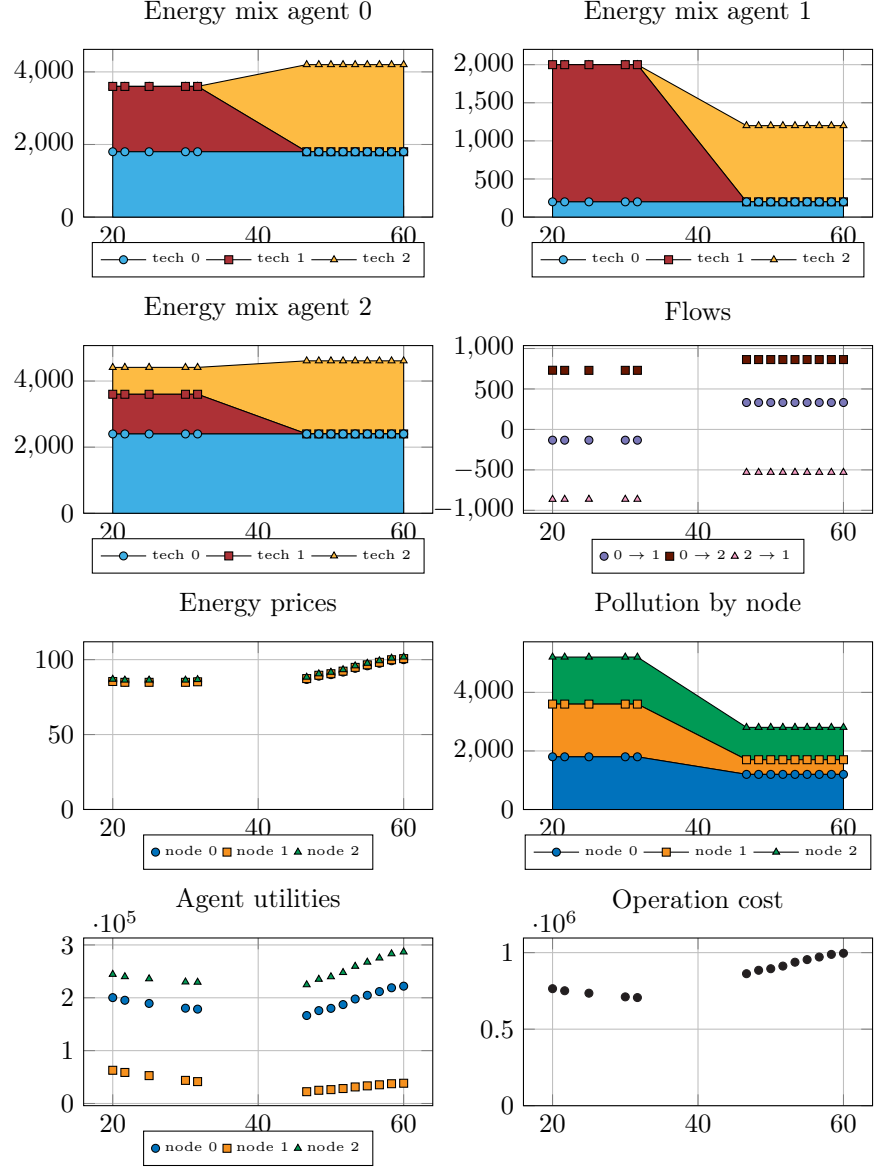

To summarize, the effects of the two pollution control mechanisms studied on market equilibria can be described as follows. The pollution cost acts as an internal penalty for agents, leading to abrupt changes in their preferences once the cost exceeds a sufficiently large threshold. Its impact on prices preserves continuity and ISO uniqueness, and no instances of multiple equilibria were observed under this mechanism. In contrast, the pollution limit induces shocks in market prices as soon as it becomes binding. This can be interpreted as conferring market power on generating agents, allowing them to extract from clean technologies the same utility or more that was previously gained from conventional sources. Nevertheless, despite allowing for multiple equilibria as in Figure \ref{fig:pollution-limit-SA-instance-2-asymmetric}, the network energy mix evolves in a gradual and controllable manner under a pollution limit. This feature may make the mechanism attractive from a policy design perspective, although the preceding discussion suggests that it should be implemented with care.

We now discuss computational complexity. Table \ref{tab:numerical-performance-instances} summarizes the performance of the grid searching strategy across experiments. Here, \(K\) denotes the number of points in which each interval \([\truecost, \maximumbid]\) was divided in order to perform the grid search. For Instance \ref{fig:network-and-table}, this value had to be lowered in order to maintain computational tractability. We then report the resulting size of the grid \(K^{\nodenumber \technumber}\) and provide statistics for three performance measures: The \textit{solution rate} refers to the proportion of initial bid configurations that resulted in a solution to the mixed complementarity problem, regardless of whether or not it was later validated as a Nash equilibrium. We also report figures for both the iterations and computation time required to obtain equilibria.

These results show a notable increase in the computational complexity of finding market equilibria as the network grows in size. We see a \textit{curse of dimensionality} effect, especially for the experiment on Instance \ref{fig:network-and-table}, where the solution rate is considerably less than for other experiments. This suggests that finding better methods for solving the underlying MCP could yield performance improvements. We also note that finding equilibria that involve an active pollution limit is significantly harder than ones where the pollution constraint is nonbinding. This can be seen from the experiments on Instance \ref{fig:instance-2}, in which the algorithm's performance on the cost experiment is similar to the one concerning Instance \ref{fig:instance-1-asymmetric}, but the metrics are notably affected once the pollution limit is introduced. Instance \ref{fig:network-and-table} remains the most challenging network to analyze as individual grid search runs may take upwards of \(20\) minutes and finding equilibria is not guaranteed.h

\begin{table}[ht]
	\centering
	\begin{tabular}{|cc|ccccc|}
		\hline
		\multicolumn{2}{|c|}{\backslashbox{Metric}{Experiment}} & \ref{fig:instance-1-asymmetric} (cost) & \ref{fig:instance-2} (cost) & \ref{fig:instance-2} (limit) & \ref{fig:instance-2-asymmetric} (limit) & \ref{fig:network-and-table} (cost) \\
		\hline
		\multicolumn{2}{|c|}{\(K\)} & 3 & 3 & 3 & 3 & 2 \\
		\hline
		\multicolumn{2}{|c|}{\(K^{\nodenumber \technumber}\)} & 9 & 81 & 81 & 81 & 512 \\
		\hline
		\multirow{3}{*}{Solution rate (\%)}
								      & Avg & 100 & 100 & 98.99 & 99.6 & 31.71 \\
								      & Min & 100 & 100 & 66.67 & 87.5 & 19.31 \\
								      & Max & 100 & 100 & 100   & 100  & 50.00    \\
								      \hline
								      \multirow{3}{*}{\makecell{Iterations for \\ first equilibrium}}
								      & Avg & 1.05 & 1.05 & 1.67 & 1.45 & 77.79 \\
								      & Min & 1    & 1    & 1    & 1    & 2     \\
								      & Max & 2    & 2    & 10   & 9    & 334 \\
								      \hline
								      \multirow{3}{*}{\makecell{Time for first \\ equilibrium (s)}}
								      & Avg & 1.47 & 2.05 & 2.20 & 1.72  & 260.57 \\
								      & Min & 1.25 & 0.81 & 0.73 & 0.74  & 6.65 \\
								      & Max & 3.06 & 4.68 & 7.55 & 13.22 & 976.38 \\
								      \hline
	\end{tabular}
	\caption{Summary of computational metrics for all performed sensitivity analysis experiments. For each metric, the average as well as the minimum and maximum values recorded during the experiment are shown. For correctness, figures are computed only for successful searches.}
	\label{tab:numerical-performance-instances}
\end{table}

\section{Conclusions and Future Work} \label{sect:conclusion}

We introduced a game-theoretic model of decentralized electricity markets with pollution control, combining a bilevel ISO-agent structure with transmission losses and emission constraints. We established well-posedness and differentiability of the ISO's response, enabling reformulations based on KKT conditions and the subsequent development of a numerical grid search strategy for equilibrium computation.

Numerical results showed that the studied environmental regulation measures can lead to complex, sometimes discontinuous, equilibrium responses. In particular, pollution penalties exhibited threshold effects, while pollution limits prompted earlier but potentially more abrupt market shifts. In both cases, strategic behavior by firms can lead to windfall profits and cost shifts onto the ISO, raising concerns for policy design in these imperfectly competitive settings.

Several modeling extensions remain open. Key directions include the incorporation of storage technologies, treatment of real-time responsiveness, and a more realistic transmission model that captures voltage and reactive power constraints. On the theoretical side, the complementarity reformulation used in this study could benefit from insights coming from constraint qualification theory as well as solvers that handle large-scale MCPs.

\section{Acknowledgements}

This work was partially supported by Centro de Modelamiento Mate\-mático (CMM) BASAL fund FB210005 for center of excellence from ANID-Chile.

\appendix  

\printbibliography

\end{document}